\documentclass[12pt]{amsart}
\usepackage[margin=1.2in]{geometry}
\usepackage{graphicx,latexsym}
\usepackage{comment}
\usepackage{tikz}
\usepackage{amsfonts, amssymb, amsmath, amsthm, , hyperref}
\usepackage{tikz}
\usetikzlibrary{matrix,arrows}
\allowdisplaybreaks

\usepackage{color}

\newcommand{\N}{\mathbb{N}}
\newcommand{\Z}{\mathbb{Z}}
\newcommand{\R}{\mathbb{R}}
\newcommand{\C}{\mathbb{C}}

\newcommand{\subsetc}{\subset_{\operatorname{comp}}}

\newcommand{\dx}{{\rm d}x }
\newcommand{\dy}{{\rm d}y }

\newcommand{\dt}{{\rm d}t }
\newcommand{\deta}{{\rm d}\eta }
\newcommand{\dxi}{{\rm d}\xi }

\newcommand{\dlambda}{{\rm d}\lambda }

\newcommand{\supp}{\operatorname{supp}}

\newtheorem{theorem}{Theorem}[section]
\newtheorem{proposition}[theorem]{Proposition}
\newtheorem{lemma}[theorem]{Lemma}
\newtheorem{corollary}[theorem]{Corollary}

\theoremstyle{definition}

\theoremstyle{remark}
\newtheorem{remark}[theorem]{Remark}

\numberwithin{equation}{section}

\makeatletter
\DeclareRobustCommand\widecheck[1]{{\mathpalette\@widecheck{#1}}}
\def\@widecheck#1#2{%
    \setbox\z@\hbox{\m@th$#1#2$}%
    \setbox\tw@\hbox{\m@th$#1%
       \widehat{%
          \vrule\@width\z@\@height\ht\z@
          \vrule\@height\z@\@width\wd\z@}$}%
    \dp\tw@-\ht\z@
    \@tempdima\ht\z@ \advance\@tempdima2\ht\tw@ \divide\@tempdima\thr@@
    \setbox\tw@\hbox{%
       \raise\@tempdima\hbox{\scalebox{1}[-1]{\lower\@tempdima\box
\tw@}}}%
    {\ooalign{\box\tw@ \cr \box\z@}}}
\makeatother

\usetikzlibrary{arrows,chains,matrix,positioning,scopes}
\makeatletter
\tikzset{join/.code=\tikzset{after node path={%
\ifx\tikzchainprevious\pgfutil@empty\else(\tikzchainprevious)%
edge[every join]#1(\tikzchaincurrent)\fi}}}
\makeatother
\tikzset{>=stealth',every on chain/.append style={join},
         every join/.style={->}}
\tikzstyle{labeled}=[execute at begin node=$\scriptstyle,
   execute at end node=$]
\begin{document}

\title[Boundary values of harmonic functions]{Quasianalytic functionals and ultradistributions as boundary values of harmonic functions}

\author[A. Debrouwere]{Andreas Debrouwere}
\thanks{A. Debrouwere was supported by  FWO-Vlaanderen through the postdoctoral grant 12T0519N}
\address{A. Debrouwere, Department of Mathematics and Data Science \\ Vrije Universiteit Brussel, Belgium\\ Pleinlaan 2 \\ 1050 Brussels \\ Belgium}
\email{Andreas.Debrouwere@vub.be}

\author[J. Vindas]{Jasson Vindas}
\thanks {J. Vindas was supported by Ghent University through the BOF-grants 01J11615 and 01J04017.}
\address{J. Vindas, Department of Mathematics: Analysis, Logic and Discrete Mathematics\\ Ghent University\\ Krijgslaan 281\\ 9000 Ghent\\ Belgium}
\email{jasson.vindas@UGent.be}

\begin{abstract}
We study boundary values of harmonic functions in spaces of quasianalytic functionals and spaces of  ultradistributions of  non-quasianalytic type. As an application, we provide a new approach to H\"ormander's support theorem for quasianalytic functionals.
Our main technical tool is a description of ultradifferentiable functions by almost harmonic functions, a concept that we introduce in this article. We work in the setting of ultradifferentiable classes defined via weight matrices. In particular, our results simultaneously apply to the two standard classes defined via weight sequences 
and via  weight functions.
\end{abstract}
\maketitle
\section{Introduction}

The representation of  functions and linear functionals as boundary values of harmonic functions is an important and useful idea in functional analysis.
For analytic functionals such a representation follows e.g.\ from Bengel's work \cite{Bengel1967} (see also \cite{Grothendieck}) on formal boundary values of zero solutions of elliptic operators and leads to an elementary proof of the support theorem for analytic functionals. Moreover, it may be used to  develop a harmonic function approach to the theory of hyperfunctions in several variables, which is reminiscent of the simple one variable theory. We refer to \cite{Hormander, Komatsu1991, Komatsu2, Schapira} for more information on this subject. In \cite{Komatsu1991}  Komatsu studied boundary values of harmonic functions in ultradistribution spaces of non-quasianalytic type (see \cite{Langenburch} for the distribution case).
  
The main goal of this article is to enhance these  results by developing a theory of boundary values of harmonic functions in spaces of \emph{quasianalytic functionals} (= compactly supported quasianalytic ultradistributions) \cite{Hormander85}. Our ideas give rise to a new  approach to the support theorem for quasianalytic functionals, originally shown by H\"ormander in \cite{Hormander85} (see also \cite{Heinrich}). H\"{or}mander's proof of this result
is quite involved. We believe that the approach given here is conceptually simpler, supplying a description of the support (= minimal carrier) of a quasianalytic functional in terms of the harmonic continuation properties of its Poisson transform. Furthermore, in the non-quasianalytic case, we obtain alternative proofs of Komatsu's results \cite{Komatsu1991}. Our method allows us to work under much weaker assumptions on the defining weight sequence (see Remark \ref{comparison} for details).

A simple but powerful method to ensure the existence of (ultra)distributional boundary values of holomorphic functions consists in combining Stokes' theorem (more precisely, the formula \cite[Equation (3.1.9), p.\ 62]{Hormander}) with the notion of almost analytic extensions. This technique was used for the first time by H\"ormander \cite[p.\ 64]{Hormander} for distributions and was later extended to the ultradistributional setting by Petzsche and Vogt \cite{P-V} (see also \cite{Petzsche1984}). We mention that the characterization of Denjoy-Carleman classes by almost analytic extensions goes back to Dyn'kin \cite{Dynkin1980, Dynkin1993}. We refer to the recent article \cite{R-S19} for the newest generalizations of such results and an overview of the topic of almost analytic extensions. Here we develop a similar method to establish the existence of ultradistributional boundary values of harmonic functions. Our method combines  Green's theorem with a novel description of ultradifferentiable functions by so-called \emph{almost harmonic functions}. Therefore, the first part of this article is devoted to an almost harmonic function characterization of  ultradifferentiable classes. 

We work with the notion of ultradifferentiability
defined via weight matrices, as introduced in \cite{R-S15}. This leads to a unified treatment of ultradifferentiable classes defined via weight sequences \cite{Komatsu} (Denjoy-Carleman approach) and via weight functions  \cite{B-M-T} (Braun-Meise-Taylor approach), but also comprises other spaces, e.g., the union and intersection of all Gevrey spaces. We point out that we shall infer the weight function case from the weight matrix case by employing the method from \cite{R-S19}, which is based upon results from \cite{R-S15,R-S16,R-S18}.

Finally, we fix some notation. Let $\Omega \subseteq \R^d$ be open. We write $K \subsetc \Omega$ to indicate that $K$ is a compact subset of $\Omega$. The notation $\Theta \Subset \Omega$ means that $\Theta$ is a relatively compact open subset of $\Omega$.  We write $\mathcal{H}(\Omega)$ for the space of harmonic functions on $\Omega$ and endow it with the compact-open topology. Points of $\R^{d+1} = \R^d \times \R$ will be denoted by $(x,y) = (x_1, \ldots, x_d,y)$. We will often identify $\R^d$ with the subspace $\R^d \times \{0\}$ of $\R^{d+1}$. If $V \subseteq \R^{d+1}$ is open and symmetric with respect to $y$, we write $\mathcal{H}_-(V)$ for the space of harmonic functions in $V$ that are odd with respect to $y$. 

\section{Ultradifferentiable classes}\label{sect-uc}
\subsection{Denjoy-Carleman classes}  Let $M = (M_p)_{p \in \N}$ be a sequence of positive numbers. We set $m_p = M_p / M_{p -1}$, $p \in \Z_+$.
Furthermore, we define $M^* = (M_p/p!)_{p \in \N}$ and $m^*_p =  M^*_p / M^*_{p -1} = m_p/p$, $p \in \Z_+$.  We will make use of the following conditions on a positive sequence $M$:
\begin{itemize}
\item [$(M.1)\ \: $]   $(m_p)_{p \in \Z_+}$ is increasing;
\item [$(M.1)^{*}\: $]  $(m^*_p)_{p \in \Z_+}$ is increasing;
\item [$(M.1)^{*}_{\mathrm{w}}$] $(m^*_p)_{p \in \Z_+}$ is almost increasing, i.e., $m^*_q \leq C m^*_p$, 
$q \leq p$, for some $C > 0$;
\item [$(M.2)'\ $] $M_{p+1}\leq C H^{p} M_p$, $p\in\mathbb{N}$, for some $C,H > 0$.
\end{itemize}
We refer to  \cite{Komatsu} for the meaning of the standard conditions $(M.1)$, $(M.2)'$, and $(M.1)^{*}$. Condition $(M.1)^{*}_{\mathrm{w}}$ is inspired by \cite[Lemma 8]{R-S18}. A sequence $M$  of positive numbers is called a \emph{weight sequence} if $M_0  = 1$, $\lim_{p \to \infty} m_p = \infty$, and $M$ satisfies $(M.1)$. 
A weight sequence $M$ is called $\emph{non-quasianalytic}$ if
$$\displaystyle \sum_{p=1}^{\infty}\frac{1}{m_{p}}<\infty$$
and $\emph{quasianalytic}$ otherwise. 
 
The relation $N\subset M$ between two weight sequences $M$ and $N$ means that there are $C,H>0$ such that 
$N_{p}\leq CH^{p}M_p,$ $p\in\mathbb{N}$. The stronger relation $N\prec M$ means that the latter inequality remains valid for every $H>0$ and a suitable $C=C_H>0$.  We write $N \approx M$ if both $N \subset M$ and $M \subset N$ hold. 

We shall also use the following condition on a weight sequence $M$:
\begin{itemize}
 \item[($NA)$] $p! \prec M$. 
\end{itemize}
Each non-quasianalytic weight sequence satisfies $(NA)$.

The \emph{associated function} of a positive sequence $M$ is defined as
$$
\omega_M(t):=\sup_{p\in\mathbb{N}}\log\frac{t^pM_0}{M_p},\qquad t \geq 0.
$$

Let $M$ be a weight sequence  and  let $\Omega \subseteq \R^d$ be open. For $h > 0$ we write $\mathcal{B}^{M,h}(\Omega)$ for the Banach space consisting of all $\varphi \in C^\infty(\Omega)$ such that
$$
\| \varphi \|_{\mathcal{B}^{M,h}(\Omega)} := \sup_{\alpha \in \N^d}\sup_{x \in \Omega} \frac{|\varphi^{(\alpha)}(x)|}{h^{|\alpha|}M_{|\alpha|}} < \infty. 
$$
We define
$$
\mathcal{B}^{(M)}(\Omega) := \varprojlim_{h \rightarrow 0^+}\mathcal{B}^{M,h}(\Omega), \qquad  \mathcal{B}^{\{M\}}(\Omega) :=  \varinjlim_{h \rightarrow \infty}\mathcal{B}^{M,h}(\Omega),
$$
and
$$
\mathcal{E}^{(M)}(\Omega) := \varprojlim_{\Theta \Subset \Omega} \mathcal{B}^{(M)}(\Theta), \qquad \mathcal{E}^{\{M\}}(\Omega) := \varprojlim_{\Theta \Subset \Omega} \mathcal{B}^{\{M\}}(\Theta).
$$
The space $\mathcal{E}^{\{p!\}}(\Omega)$ coincides with the space $\mathcal{A}(\Omega)$ of real analytic functions in $\Omega$.

From now on we shall write $[M]$ instead of $(M)$ or $\{M\}$ if we want to treat both cases simultaneously. In addition, we shall often first state assertions for the Beurling case (= $(M)$-case) followed in parentheses by the corresponding ones for the Roumieu case (= $\{M\}$-case).

Given two weight sequences $M$ and $N$, we have that $\mathcal{E}^{[N]}(\Omega) \subseteq \mathcal{E}^{[M]}(\Omega)$ continuously if $N \subset M$ and $\mathcal{E}^{\{N\}}(\Omega) \subset \mathcal{E}^{(M)}(\Omega)$ continuously if $N \prec M$. Hence, $\mathcal{E}^{[N]}(\Omega) = \mathcal{E}^{[M]}(\Omega)$ as locally convex spaces if $N \approx M$ and  $\mathcal{A}(\Omega) \subset \mathcal{E}^{[M]}(\Omega)$ continuously if $M$ satisfies $(NA)$.

Let $M$ be a weight sequence. For $K \subsetc \R^d$ and $h >0$ we write $\mathcal{D}^{M,h}_K$ for the Banach space consisting of all $\varphi \in C^\infty(\R^d)$ with $\operatorname{supp} \varphi \subseteq K$ such that $\| \varphi \|_{\mathcal{B}^{M,h}(\R^d)} < \infty$. We set
$$
\mathcal{D}^{(M)}_K := \varprojlim_{h \rightarrow 0^+}\mathcal{D}^{M,h}_K, \qquad  \mathcal{D}^{\{M\}}_K :=  \varinjlim_{h \rightarrow \infty}\mathcal{D}^{M,h}_K.
$$
For $\Omega \subseteq \R^d$ open we define 
$$
\mathcal{D}^{[M]}(\Omega) := \varinjlim_{K \subsetc \Omega} \mathcal{D}^{[M]}_K.
$$
The space $\mathcal{D}^{[M]}(\Omega)$ is non-trivial if and only if $M$ is non-quasianalytic. 

\subsection{Classes defined by weight matrices} \label{wm-sect}
Following \cite{R-S19} (see also \cite{R-S15}), we define a \emph{weight matrix} as a non-empty family $\mathfrak{M}$ of weight sequences that is totally ordered with respect to the pointwise order relation $\leq$ on sequences.  We will make use of the following conditions on a weight matrix $\mathfrak{M}$:
\begin{itemize}

\item [$(\mathfrak{M}.1)^{*}_{\mathrm{w}}$] $\forall M \in \mathfrak{M} \, \exists N \in \mathfrak{M} \, \exists C > 0 \, \forall p \in \Z_+ \, \forall 1 \leq q \leq p \, : \, n^*_q \leq C m^*_p$;
\item [$\{\mathfrak{M}.1\}^{*}_{\mathrm{w}}$] $\forall M \in \mathfrak{M} \, \exists N \in \mathfrak{M} \, \exists C > 0 \, \forall p \in \Z_+ \, \forall 1 \leq q \leq p \, : \, m^*_q \leq C n^*_p$;
\item [$(\mathfrak{M}.2)' \ $] $\forall M \in \mathfrak{M} \, \exists N \in \mathfrak{M} \, \exists C,H > 0 \, \forall p \in \N \, : \, N_{p+1} \leq CH^pM_p$; 
\item [$\{\mathfrak{M}.2\}' \ $] $\forall M \in \mathfrak{M} \, \exists N \in \mathfrak{M} \, \exists C,H > 0 \, \forall p \in \N \, : \, M_{p+1} \leq CH^pN_p$; 
\item [$(NA)\ \ $]  Each $M \in \mathfrak{M}$ satisfies $(NA)$.
\end{itemize}
The conditions $(\mathfrak{M}.2)'$ and $\{\mathfrak{M}.2\}'$  are denoted by $(\mathfrak{M}_{(\operatorname{dc})})$ and $(\mathfrak{M}_{\{\operatorname{dc}\}})$ in \cite{R-S15}, respectively. The conditions $(\mathfrak{M}.1)^{*}_{\mathrm{w}}$ and $\{\mathfrak{M}.1\}^{*}_{\mathrm{w}}$ were introduced in \cite{R-S18} but no name was given to them there. A weight matrix $\mathfrak{M}$ is said to be non-quasianalytic if each $M \in \mathfrak{M}$ is non-quasianalytic. 
 
The relation $\mathfrak{N} (\subset) \mathfrak{M}$ ($\mathfrak{N} \{\subset\}\mathfrak{M}$) between two weight matrices $\mathfrak{M}$ and $\mathfrak{N}$ means that
$$
\forall M \in \mathfrak{M} \, \exists N \in \mathfrak{N} \, : \, N \subset M  \, (\forall M \in \mathfrak{N} \, \exists N \in \mathfrak{M} \, : \, M \subset N).
$$
 We write $\mathfrak{N} [\approx] \mathfrak{M}$ if both $\mathfrak{N} [\subset] \mathfrak{M}$ and $\mathfrak{M} [\subset] \mathfrak{N}$ hold. Furthermore, we define the relation $\mathfrak{N} \prec \mathfrak{M}$ as 
 $$
\forall M \in \mathfrak{M} \, \forall N \in \mathfrak{N} \, : \, N \prec M.  
$$

\begin{lemma} \label{equivalent} Let  $\mathfrak{M}$ be a weight matrix satisfying $[\mathfrak{M}.1]^{*}_{\mathrm{w}}$ and $[\mathfrak{M}.2]'$. Then, there is a weight matrix $\mathfrak{N}$ with $\mathfrak{M} [\approx] \mathfrak{N}$ such that $\mathfrak{N}$ satisfies $[\mathfrak{M}.2]'$ and each  $N \in \mathfrak{N}$ satisfies $(M.1)^{*}$. If  $\mathfrak{M}$ satisfies $(NA)$ (is non-quasianalytic, respectively), then $\mathfrak{N}$ can be chosen in such a way that $\mathfrak{N}$ satisfies $(NA)$ (is non-quasianalytic, respectively) as well. 
\end{lemma}
\begin{proof}
By \cite[Lemma 8]{R-S18}, there is a weight matrix $\widetilde{\mathfrak{N}}$  with $\mathfrak{M} [\approx] \widetilde{\mathfrak{N}}$ such that each  $N \in \widetilde{\mathfrak{N}}$ satisfies $(M.1)^{*}$. Since the condition $[\mathfrak{M}.2]'$ is stable under the relation $[\approx]$, $\widetilde{\mathfrak{N}}$ also satisfies $[\mathfrak{M}.2]'$. Suppose that $\mathfrak{M}$ satisfies $(NA)$ (is non-quasianalytic, respectively). In the Beurling case, it is clear that $\mathfrak{N} = \widetilde{\mathfrak{N}}$ also satisfies $(NA)$ (is non-quasianalytic, respectively). In the Roumieu case, there exists $N_0 \in \widetilde{\mathfrak{N}}$ that satisfies $(NA)$ (is non-quasianalytic, respectively). The result then holds for $\mathfrak{N} = \{ N \in \widetilde{\mathfrak{N}} \, | \, N_0 \leq N \}$.
\end{proof}
Let $\mathfrak{M}$ be a weight matrix and let $\Omega \subseteq \R^d$ be open. We define
$$
\mathcal{B}^{(\mathfrak{M})}(\Omega) := \varprojlim_{M \in \mathfrak{M}}\mathcal{B}^{(M)}(\Omega), \qquad  \mathcal{B}^{\{\mathfrak{M}\}}(\Omega) :=  \varinjlim_{M \in \mathfrak{M}}\mathcal{B}^{\{M\}}(\Omega),
$$
 and
$$
\mathcal{E}^{[\mathfrak{M}]}(\Omega) := \varprojlim_{\Theta \Subset \Omega}\mathcal{B}^{[\mathfrak{M}]}(\Theta).
$$

Given two weight matrices $\mathfrak{M}$ and $\mathfrak{N}$, we have that $\mathcal{E}^{[\mathfrak{N}]}(\Omega) \subseteq \mathcal{E}^{[\mathfrak{M}]}(\Omega)$ continuously if $\mathfrak{N} [\subset] \mathfrak{M}$ and $\mathcal{E}^{\{\mathfrak{N}\}}(\Omega) \subset \mathcal{E}^{(\mathfrak{M})}(\Omega)$ continuously if $\mathfrak{N} \prec \mathfrak{M}$. Hence, $\mathcal{E}^{[\mathfrak{N}]}(\Omega) = \mathcal{E}^{[\mathfrak{M}]}(\Omega)$ as locally convex spaces if $N [\approx] M$ and $\mathcal{A}(\Omega) \subset \mathcal{E}^{[\mathfrak{M}]}(\Omega)$ continuously if $\mathfrak{M}$ satisfies $(NA)$. 

Let $\mathfrak{M}$ be a weight matrix and let $K \subsetc \R^d$.  We set
$$
\mathcal{D}^{(\mathfrak{M})}_K := \varprojlim_{M \in \mathfrak{M}}\mathcal{D}^{(M)}_K, \qquad  \mathcal{D}^{\{\mathfrak{M}\}}_K :=  \varinjlim_{M \in \mathfrak{M}}\mathcal{D}^{\{M\}}_K.
$$
For $\Omega \subseteq \R^d$ open we define 
$$
\mathcal{D}^{[\mathfrak{M}]}(\Omega) := \varinjlim_{K \subsetc \Omega} \mathcal{D}^{[\mathfrak{M}]}_K.
$$
The space $\mathcal{D}^{(\mathfrak{M})}(\Omega)$ is non-trivial if and only if  $\mathfrak{M}$ is non-quasianalytic \cite[Theorem 4.1]{Schindl}. In fact, if $\mathfrak{M}$ is non-quasianalytic, there is a non-quasianalytic weight sequence $N$ such that $N \prec M$ for all $M \in \mathfrak{M}$ and, thus, $\{0 \} \subsetneq \mathcal{D}^{\{N\}}(\Omega) \subset \mathcal{D}^{(\mathfrak{M})}(\Omega)$ \cite[Proposition 4.7]{Schindl}. It is clear that $\mathcal{D}^{\{\mathfrak{M}\}}(\Omega)$ is non-trivial if and only if  there exists a non-quasianalytic $M \in \mathfrak{M}$. In such a case, we can find a non-quasianalytic weight matrix $\mathfrak{N} \subseteq \mathfrak{M}$ such that $\mathfrak{M} \{\approx\} \mathfrak{N}$ and, thus, $\mathcal{D}^{\{\mathfrak{M}\}}(\Omega) = \mathcal{D}^{\{\mathfrak{N}\}}(\Omega)$.

\begin{remark}\label{countable}
The name weight matrix is justified by the fact that for every weight matrix $\mathfrak{M}$ there is a countable weight matrix $\mathfrak{N} \subseteq \mathfrak{M}$ such that $\mathfrak{M} [\approx] \mathfrak{N}$ (cf.\ the proof of \cite[Lemma 2.5]{R-S19}).
\end{remark}

\subsection{Braun-Meise-Taylor classes}   By a \emph{weight function}  we mean a continuous increasing function $\omega: [0,\infty) \rightarrow [0,\infty)$ with $\omega_{|[0,1]} \equiv 0$ satisfying the following properties:	\begin{itemize}
		\item[$(\alpha)$] $\omega(2t) = O(\omega(t))$ as $t \rightarrow \infty$;
		\item[$(\beta)$] $\omega(t) = O(t)$ as $t \rightarrow \infty$;
		\item[$(\gamma)$] $\log t = o(\omega(t))$ as $t \rightarrow \infty$;
		\item[$(\delta) $] $\phi = \phi_\omega: [0, \infty) \rightarrow [0, \infty)$, $\phi(t) = \omega(e^{t})$, is convex.
	\end{itemize}
We refer to \cite{B-M-T} for the meaning of these conditions. A weight function $\omega$ is called non-quasianalytic if
$$\int_{0}^{\infty}\frac{\omega(t)}{1+t^2} \dt < \infty$$
and quasianalytic otherwise.  Each non-quasianalytic weight function $\omega$ satisfies $\omega(t) = o(t)$.
We also consider the following condition on a weight function $\omega$ :
\begin{itemize}
\item[$(\alpha_0)$] $\exists C > 0 \, \exists t_0 > 0 \, \forall \lambda \geq 1 \, \forall t \geq t_0 \, : \, \omega(\lambda t) \leq C\lambda \omega(t).$
\end{itemize}
By \cite[Theorem 6.3]{R-S15} (see also the proof of \cite[Proposition 1.1]{P-V}), a weight function $\omega$ satisfies $(\alpha_0)$ if and only if there is a concave weight function $\sigma$ such that $\omega \asymp \sigma$ (meaning that $\omega(t) = O(\sigma(t))$ and $\sigma(t) = O(\omega(t))$).
 
Let $\omega$ be a weight function. We define 
$$
\phi^{*} : [0, \infty) \rightarrow [0, \infty), \, \phi^{*}(t) = \sup_{r \geq 0} \{tr- \phi(r)\}.
$$
The function $\phi^{*}$ is increasing and convex, $\phi^*(0) = 0$, $(\phi^*)^* = \phi$, and $\phi^*(t)/t \nearrow \infty$ on  $[0,\infty)$. 

Let $\omega$ be a weight function and let $\Omega \subseteq \R^d$ be open. For $h > 0$ we write $\mathcal{B}^{\omega,h}(\Omega)$ for the Banach space consisting of all $\varphi \in C^\infty(\Omega)$ such that
$$
\| \varphi \|_{\mathcal{B}^{\omega,h}(\Omega)} := \sup_{\alpha \in \N^d}\sup_{x \in \Omega} |\varphi^{(\alpha)}(x)| \exp\left(-\frac{1}{h}\phi^*(h|\alpha|) \right)<\infty. 
$$
We define
$$
\mathcal{B}^{(\omega)}(\Omega) := \varprojlim_{h \rightarrow 0^+}\mathcal{B}^{\omega,h}(\Omega), \qquad  \mathcal{B}^{\{\omega\}}(\Omega) :=  \varinjlim_{h \rightarrow \infty}\mathcal{B}^{\omega,h}(\Omega),
$$
and 
$$
\mathcal{E}^{[\omega]}(\Omega) := \varprojlim_{\Theta \Subset \Omega} \mathcal{B}^{[\omega]}(\Theta).
$$
For $\omega(t) = \max \{ t-1, 0\}$ the space $\mathcal{E}^{\{\omega\}}(\Omega)$ coincides with $\mathcal{A}(\Omega)$.

Given two weight functions $\omega$ and $\sigma$, we have that $\mathcal{E}^{[\sigma]}(\Omega) \subseteq \mathcal{E}^{[\omega]}(\Omega)$ continuously if $\omega(t) = O(\sigma(t))$ and $\mathcal{E}^{\{\sigma\}}(\Omega) \subset \mathcal{E}^{(\omega)}(\Omega)$ continuously if $\omega(t) = o(\sigma(t))$. Hence, $\mathcal{E}^{[\omega]}(\Omega) = \mathcal{E}^{[\sigma]}(\Omega)$ as locally convex spaces if $\omega \asymp \sigma$ and $\mathcal{A}(\Omega) \subset \mathcal{E}^{[\omega]}(\Omega)$ continuously if $\omega(t) = o(t)$.

Let $\omega$ be a weight function. For $K \subsetc \R^d$ and $h >0$ we write $\mathcal{D}^{\omega,h}_K$ for the Banach space consisting of all $\varphi \in C^\infty(\R^d)$ with $\operatorname{supp} \varphi \subseteq K$ such that $\| \varphi \|_{\mathcal{B}^{\omega,h}(\R^d)} < \infty$. We set
$$
\mathcal{D}^{(\omega)}_K := \varprojlim_{h \rightarrow 0^+}\mathcal{D}^{\omega,h}_K, \qquad  \mathcal{D}^{\{\omega\}}_K :=  \varinjlim_{h \rightarrow \infty}\mathcal{D}^{\omega,h}_K.
$$
For $\Omega \subseteq \R^d$ open we define 
$$
\mathcal{D}^{[\omega]}(\Omega) := \varinjlim_{K \subsetc \Omega} \mathcal{D}^{[\omega]}_K.
$$
The space $\mathcal{D}^{[\omega]}(\Omega)$ is non-trivial if and only if $\omega$ is non-quasianalytic.

Given a weight function $\omega$, we associate to it the weight matrix $\mathfrak{M}_{\omega} = (M_{\omega}^{h})_{h > 0}$, where the weight sequence $M_\omega^{h}= (M^h_{\omega,p})_{p \in \N}$ is defined by
$$
M^h_{\omega,p} := \exp\left(\frac{1}{h}\phi^*(hp) \right), \qquad p \in \N.
$$
We have that $ \omega \asymp \omega_M$ for each $M \in \mathfrak{M}_\omega$ \cite[Lemma 5.7]{R-S15}. Hence, by \cite[Lemma 3.10]{Komatsu}, $\omega(t) = o(t)$ if and only if  $\mathfrak{M}_{\omega}$ satisfies $(NA)$. Similarly, \cite[Lemma 4.1]{Komatsu} yields that $\omega$ is non-quasianalytic if and only if $\mathfrak{M}_{\omega}$ is non-quasianalytic. 
The next two lemmas will enable us to infer our results for the weight function case from those  for the weight matrix case.

\begin{lemma} \label{reduction} Let $\omega$ be a weight function.
\begin{itemize}
\item[$(i)$] \cite[Corollary 5.15]{R-S15} $\mathcal{E}^{[\omega]}(\Omega) = \mathcal{E}^{[\mathfrak{M}_\omega]}(\Omega)$ as locally convex spaces for all $\Omega \subseteq \R^d$ open.
\item[$(ii)$] \cite[Proposition 4.5]{R-S19} If $\omega$ satisfies $(\alpha_0)$ and $\omega(t) = o(t)$, there is a weight matrix $\mathfrak{M}$ with $\mathfrak{M}_\omega [\approx] \mathfrak{M}$ such that  $\mathfrak{M}$ satisfies $[\mathfrak{M}.1]^{*}_{\mathrm{w}}$, $[\mathfrak{M}.2]'$, and $(NA)$.
\end{itemize}
\end{lemma}
For  a weight function $\omega$ satisfying $\omega(t) = o(t)$ we define 
$$
\omega^{\star} : (0, \infty) \rightarrow [0, \infty), \,\omega^{\star}(s) = \sup_{t \geq 0} \{\omega(t) - ts\}.
$$
The function $\omega^{\star}$ is decreasing and convex. Given $h > 0$, we set $\mbox{}_h\omega(t) = h\omega(t)$ and $\omega_h(t) = \omega(ht)$. Then,
\begin{equation}
(\mbox{}_h\omega)^\star(s) = h\omega^\star\left ( \frac{s}{h}\right), \qquad (\omega_h)^\star(s) = \omega^\star\left ( \frac{s}{h}\right), \qquad s > 0.
\label{star-transf}
\end{equation}
For a weight sequence $M$ satisfying $(NA)$ it holds that  \cite[Proof of Lemma 5.6]{P-V} (cf. \cite[Lemma 3.10]{R-S16})
\begin{equation}
\omega^\star_M(s) \leq \omega_{M^*}  \left( \frac{1}{s}\right) \leq \omega^\star_M  \left( \frac{s}{e}\right), \qquad s > 0.
\label{star-ws}
\end{equation}
\begin{lemma} \label{reduction-1} Let $\omega$ be a weight function satisfying $\omega(t) = o(t)$. 
\begin{itemize}
\item[$(i)$] For all $M \in \mathfrak{M}_\omega$ and $h > 0$ there are $C, k > 0$ such that
\begin{equation}
\omega_{M^*} \left(\frac{1}{hs}\right) \leq \frac{1}{k} \omega^\star(ks) + \log C, \qquad s > 0.
\label{red-1}
\end{equation}
\item[$(ii)$] For all $k >0$ there are  $M \in \mathfrak{M}_\omega$ and $C,h > 0$   such that \eqref{red-1} holds.
\item[$(iii)$] For all $M \in \mathfrak{M}_\omega$ and $h > 0$ there are $C, k > 0$ such that
\begin{equation}
\frac{1}{k} \omega^\star(ks)  \leq \omega_{M^*} \left(\frac{1}{hs}\right) + \log C, \qquad s > 0.
\label{red-2}
\end{equation}
\item[$(iv)$] For all $k >0$ there are  $M \in \mathfrak{M}_\omega$ and $C,h > 0$  such that \eqref{red-2} holds.
\end{itemize}
\end{lemma}
\begin{proof}
By \eqref{star-transf} and \eqref{star-ws}, \eqref{red-1} holds if 
\begin{equation}
\omega_M \left( \frac{et}{h} \right) \leq \frac{1}{k} \omega(t) + \log C, \qquad t \geq 0,
\label{red-1-2}
\end{equation}
while 
\eqref{red-2} holds if 
\begin{equation}
\frac{1}{k} \omega(t) \leq \omega_M \left( \frac{t}{h} \right) + \log C, \qquad t \geq 0.
\label{red-2-2}
\end{equation}
Since $ \omega \asymp \omega_M$ for each $M \in \mathfrak{M}_\omega$, condition $(\alpha)$ implies that for all  $M \in \mathfrak{M}_\omega$ and $h > 0$ there are $C, k > 0$ such that \eqref{red-1-2} and \eqref{red-2-2} hold. This shows $(i)$ and $(iii)$. Next, note that
$$
\omega_{M^{k}_\omega}(t) = \sup_{p \in \N} \{ p \log t - \frac{1}{k} \phi^\ast(kp) \} \leq \sup_{r \geq 0} \{ r \log t - \frac{1}{k} \phi^\ast(kr) \} = \frac{1}{k}\omega(t)
$$
for all $k >0$. Hence, for each $k >0$, we have that  \eqref{red-1-2} actually holds with $M = M^k_\omega \in \mathfrak{M}_\omega$, $h = e$, and $C = 1$. This shows $(ii)$. Finally, note that 
$$
M^{h}_{\omega, p+q} \leq M^{2h}_{\omega,p}M^{2h}_{\omega,q}, \qquad p,q \in \N,
$$
for all $h > 0$. By \cite[Lemma 3.12]{R-S16}, we have
$$
2\omega_{M^{2h}_\omega}(t) \leq \omega_{M^{h}_\omega}(t), \qquad t \geq 0,
$$
for all $h > 0$. Hence,
$$
2^n\omega_{M^{1}_\omega}(t) \leq \omega_{M^{1/2^n}_\omega}(t), \qquad t \geq 0,
$$
for all $n \in \Z_+$. Since  $\omega \asymp \omega_{M^1_\omega}$, the latter inequality implies that for all $k >0$ there are  $M \in \mathfrak{M}_\omega$ and $C > 0$  such that \eqref{red-2-2} holds with $h = 1$. This shows $(iv)$.
\end{proof}
\section{Ultradifferentiable classes via almost harmonic functions}\label{sect-ah}
Let $\Omega \subseteq \R^d$ be open and let $\varphi_0,\varphi_1: \Omega \rightarrow \C$. The Cauchy-Kovalevski theorem implies that $\varphi_0, \varphi_1 \in \mathcal{A}(\Omega)$ if and only if for all $\Theta \Subset \Omega$ there exist $V \subseteq \R^{d+1}$ open with $V \cap \R^d = \Theta$ and $\Phi \in \mathcal{H}(V)$ such that $\Phi_{|\Theta} = \varphi_{0|\Theta}$ and $\partial_y \Phi_{|\Theta} = \varphi_{1|\Theta}$. The goal of this section is to characterize the classes $\mathcal{E}^{[\mathfrak{M}]}(\Omega)$ and $\mathcal{D}^{[\mathfrak{M}]}(\Omega)$ in  a similar way by \emph{almost harmonic functions}.  Namely, we shall show the following two results.

\begin{theorem}\label{main-part-1}
Let $\mathfrak{M}$ be a weight matrix satisfying $[\mathfrak{M}.1]^{*}_{\mathrm{w}}$, $[\mathfrak{M}.2]'$, and $(NA)$. Let $\Omega \subseteq \R^d$ be open and let $\varphi_0,\varphi_1: \Omega \rightarrow \C$. Then, $\varphi_0,\varphi_1 \in \mathcal{E}^{[\mathfrak{M}]}(\Omega)$ if and only if for all $\Theta \Subset \Omega$ and for all $M \in \mathfrak{M}$, $h > 0$ (for some $M \in \mathfrak{M}$, $h > 0$) the following holds: For some/all $V \subseteq \R^{d+1}$ open with $V \cap \R^d = \Theta$ there exists $\Phi \in C^2(V)$ such that $\Phi_{|\Theta} = \varphi_{0|\Theta}$, $\partial_y \Phi_{|\Theta} = \varphi_{1|\Theta}$, and
$$
 \sup_{(x,y) \in V} |\Delta\Phi(x,y)| e^{\omega_{M^*}\left(\frac{1}{h|y|}\right)}< \infty.
 $$
\end{theorem}
\begin{theorem}\label{main-part-2}
Let $\mathfrak{M}$ be a non-quasianalytic weight matrix satisfying $[\mathfrak{M}.1]^{*}_{\mathrm{w}}$ and $[\mathfrak{M}.2]'$. Let $\Omega \subseteq \R^d$ be open and let $V \subseteq \R^{d+1}$ be open such that $V \cap \R^d = \Omega$. Let $\varphi_0,\varphi_1: \Omega \rightarrow \C$. Then, $\varphi_0,\varphi_1 \in \mathcal{D}^{[\mathfrak{M}]}(\Omega)$ if and only if for all $M \in \mathfrak{M}$, $h > 0$ (for some $M \in \mathfrak{M}$, $h > 0$)  there exists $\Phi \in C^2_c(V)$ such that $\Phi_{|\Omega} = \varphi_{0}$, $\partial_y \Phi_{|\Omega} = \varphi_{1}$, and
$$
 \sup_{(x,y) \in V} |\Delta\Phi(x,y)| e^{\omega_{M^*}\left(\frac{1}{h|y|}\right)}< \infty.
 $$
\end{theorem}

The proofs of Theorem \ref{main-part-1} and Theorem \ref{main-part-2}  are divided into several intermediate results.
\begin{proposition}\label{almost-harmonic} 
Let $M$, $N$, and $Q$ be three weight sequences satisfying $(NA)$. Suppose that $Q$ satisfies $(M.1)^{*}$ and
\begin{equation}
M_{p+2} \leq C_0H^p_0Q_p,  \qquad Q_{p+2} \leq C_1H^p_1N_p, \qquad p \in \N,
\label{M2'double}
\end{equation}
for some $C_0,H_0,C_1, H_1 > 0$.  Then, there is $A > 0$ such that for all $\Theta \subseteq \R^d$  open and $h >0$ the following holds: For all $\varphi_0, \varphi_1 \in \mathcal{B}^{M,h}(\Theta)$ there exists $\Phi = \Phi(\varphi_0, \varphi_1) \in C^2(\Theta \times \R)$ such that 
\begin{itemize}
\item[$(i)$] $\Phi_{| \Theta} = \varphi_0$ and $\partial_y \Phi_{| \Theta} = \varphi_1$;
\item[$(ii)$] $\displaystyle ||| \Phi ||| := \sup_{(x,y) \in \Theta \times \R}  |\Delta\Phi(x,y)| e^{\omega_{N^*}\left(\frac{1}{Ah|y|}\right)} < \infty$.
\end{itemize}
Moreover, there is $C > 0$ such that for all $\varphi_0,\varphi_1 \in \mathcal{B}^{M,h}(\Theta)$
\begin{equation}
\max\{ ||| \Phi(\varphi_0, \varphi_1) |||, \max_{\substack{ \alpha \in \N^{d+1}; \\ |\alpha| \leq 1}} \| \partial^{\alpha}\Phi(\varphi_0, \varphi_1)\|_{L^\infty(\Theta\times \R)} \} \leq C\max_{j = 0,1} \{ \| \varphi_j\|_{\mathcal{B}^{M,h}(\Theta)} \}.\label{sobolev-bounds}
\end{equation}
\end{proposition}
For $\varphi_0, \varphi_1 \in \mathcal{B}^{\{p!\}}(\Theta)$ the series (cf.\ \cite[p.\ 330]{Hormander})
\begin{equation}
\Phi_0(x,y) = \sum_{p=0}^\infty \frac{y^{2p}}{(2p)!} (-\Delta)^p \varphi_0(x), \qquad \Phi_1(x,y) = \sum_{p=0}^\infty \frac{y^{2p+1}}{(2p+1)!} (-\Delta)^p \varphi_1(x),
\label{harmonic-series}
\end{equation}
are absolutely convergent in some open subset $V$ of $\R^{d+1}$ with $V \cap \R^d = \Theta$ and $\Phi = \Phi_0 + \Phi_1$ is a harmonic function on $V$ such that $\Phi_{| \Theta} = \varphi_0$ and $\partial_y \Phi_{| \Theta} = \varphi_1$. The idea of the proof of Proposition \ref{almost-harmonic} is to suitably modify the series in \eqref{harmonic-series}. This approach is inspired by Petzsche's construction of almost analytic extensions by means of modified Taylor series \cite[Proposition 2.2]{Petzsche1984}.
Furthermore, in our estimates we follow the same technique as in \cite[Proposition 3.12]{R-S19}, which is essentially due to Dyn'kin \cite{Dynkin1980, Dynkin1993}.

\begin{proof}[Proof of Proposition \ref{almost-harmonic}]  
Pick  $\chi \in \mathcal{D}(\R)$  such that $\operatorname{supp} \chi \subseteq [-2,2]$ and $\chi \equiv 1$ on $[-1,1]$. Set $\mu = 2 \sqrt{2d}H_0$. 
Let $j =0,1$.
For $\varphi \in \mathcal{B}^{M,h}(\Theta)$ we  define
\begin{equation}
\Phi_j(x,y) =  \Phi_j(\varphi)(x,y) =	\sum_{p=0}^\infty \frac{y^{2p +j}}{(2p +j)!} (-\Delta)^p \varphi(x) \chi(\mu h q^*_{2p+j} y), \quad (x,y) \in \Theta \times \R.
\label{formula-P}
\end{equation}
Since $q^*_p \nearrow \infty$, the above series is finite on $\Theta \times\{ y \in \R \, | \, |y| \geq \varepsilon \}$ for each $\varepsilon > 0$. Hence, $\Phi_j \in C^\infty(\Theta \times (\R \backslash \{0\}))$. 
Set $A = 2 \sqrt{2d} H_0H_1= \mu H_{1}$.
We claim that:
\begin{equation}
\limsup_{y \to 0} \sup_{x \in \Theta} |\Delta \Phi_j(x,y)| e^{\omega_{N^*}\left(\frac{1}{Ah|y|}\right)} < \infty,
\label{tech-0}
\end{equation}
\begin{equation}
\lim_{y \to 0} \partial^\alpha_x \Phi_j(x,y) = \delta_{j,0} \partial^\alpha \varphi(x) \mbox{ uniformly for $x \in \Theta$, $\alpha \in \N^d, |\alpha| \leq 2$},
\label{tech-1}
\end{equation}
\begin{equation}
\lim_{y \to 0} \partial_y  \partial^\alpha_x  \Phi_j(x,y) = \delta_{j,1} \partial^\alpha \varphi(x) \mbox{ uniformly for $x \in \Theta$, $\alpha \in \N^d, |\alpha| \leq 1$},
\label{tech-2}
\end{equation}
where $\delta_{j,k}$ denotes the Kronecker delta. These properties imply that 
\begin{equation}
\Phi(\varphi_0, \varphi_1) = \Phi_0(\varphi_0) + \Phi_1(\varphi_1), \qquad \varphi_0, \varphi_1 \in \mathcal{B}^{M,h}(\Theta),
\label{formula-P-1}
\end{equation}
belongs to $C^2(\Theta \times \R)$ and satisfies $(i)$ and $(ii)$. 
We now prove the above claims. In the rest of the proof $C$ will denote a  positive constant that is independent of $\varphi$  but may vary from place to place. We introduce the following auxiliary function
$$
 \Gamma(t) = \min \left\{ p \in \N \, | \, q^*_{p+1} \geq \frac{1}{t} \right\}, \qquad  0 < t \leq \frac{1}{q^*_1}.
 $$
Fix $0 < t \leq 1/q^*_1$. Then, $p \leq \Gamma(t)$ if and only if $tq^*_p < 1$ for all $p \in \Z_+$. 
Hence, the function $p \mapsto t^pQ^{*}_p$ is decreasing for $p \leq \Gamma(t)$ and increasing for $p \geq \Gamma(t)$. Consequently,  $t^{\Gamma(t)} Q^{*}_{\Gamma(t)} = e^{-\omega_{Q^{*}}\left (\frac{1}{t} \right)}$. 
We start by showing \eqref{tech-0}. Note that $\Delta \Phi_j = S_1 + S_2  + S_3$, where
\begin{gather*}
S_1(x,y) = \sum_{p=0}^\infty \frac{y^{2p+j}}{(2p+j)!} (-\Delta)^{p+1} \varphi(x) (\chi(\mu h q^*_{2p+ 2 +j} y) - \chi(\mu h q^*_{2p+j} y)),\\
S_2 (x,y)= 2\sum_{p=1-j}^\infty \frac{y^{2p+j-1}}{(2p+j-1)!} (-\Delta)^{p} \varphi(x) \mu h q^*_{2p+j} \chi'(\mu h q^*_{2p+ j} y),\\
S_3(x,y) = \sum_{p=0}^\infty \frac{y^{2p+j}}{(2p+j)!} (-\Delta)^{p} \varphi(x) ( \mu h q^*_{2p+j})^2\chi''( \mu h q^*_{2p+j} y).
\end{gather*}
 For all $(x,y) \in \Theta \times (\R \backslash \{0\}))$ with $|y|$ small enough, we have that
\begin{align*}
 | S_1(x,y)| &\leq  \sum_{\Gamma(\mu h|y|) - 2 < 2p + j \leq \Gamma(\mu h|y|/2)}  \frac{|y|^{2p+j}}{(2p+j)!} |\Delta^{p+1} \varphi(x)| |\chi(\mu h q^*_{2p+ 2 +j} y) - \chi(\mu h q^*_{2p+j} y)| \\
& \leq C \|\varphi\|_{\mathcal{B}^{M,h}(\Theta)}\sum_{\Gamma(\mu h|y|) - 2 < 2p + j \leq \Gamma(\mu h|y|/2)}\frac{|y|^{2p+j}}{(2p+j)!} (\sqrt{d}h)^{2p+2} M_{2p+2}  \\
& \leq C \|\varphi\|_{\mathcal{B}^{M,h}(\Theta)}\sum_{\Gamma(\mu h|y|) - 2 < 2p + j \leq \Gamma(\mu h|y|/2)} \frac{1}{2^p}(\mu h|y|/2)^{2p+j} Q^*_{2p+j}\\
& \leq C \|\varphi\|_{\mathcal{B}^{M,h}(\Theta)} (\mu h|y|/2)^{\Gamma(\mu h|y|) - 2} Q^*_{\Gamma(\mu h|y|) }\\
& \leq C \|\varphi\|_{\mathcal{B}^{M,h}(\Theta)} (\mu h|y|)^{-2} e^{-\omega_{Q^*}\left (\frac{1}{\mu h|y|} \right)}\\
& \leq C \|\varphi\|_{\mathcal{B}^{M,h}(\Theta)}  e^{-\omega_{N^*}\left (\frac{1}{A h |y|} \right)}.
\end{align*}
Likewise, for all $(x,y) \in \Theta \times (\R \backslash \{0\}))$ with $|y|$ small enough, one gets
\begin{align*}
| S_2(x,y)| &\leq C\sum_{\Gamma(\mu h|y|) < 2p +j \leq \Gamma(\mu h|y|/2)} \frac{|y|^{2p+j -1}}{(2p+j-1)!} |\Delta^{p} \varphi(x)|  q^*_{2p+j} |\chi'(\mu h q^*_{2p+j} y)| \\
& \leq C \|\varphi\|_{\mathcal{B}^{M,h}(\Theta)} \sum_{\Gamma(\mu h|y|) < 2p + j \leq \Gamma(\mu h|y|/2)} \frac{|y|^{2p+j -2}}{(2p+j-2)!}  (\sqrt{d}h)^{2p}M_{2p} \\
& \leq C \|\varphi\|_{\mathcal{B}^{M,h}(\Theta)}\sum_{\Gamma(\mu h|y|)  < 2p + j \leq \Gamma(\mu h|y|/2)} \frac{1}{2^p}(\mu h|y|/2)^{2p+j-2} Q^*_{2p+j-2}\\
& \leq C \|\varphi\|_{\mathcal{B}^{M,h}(\Theta)} (\mu h|y|/2)^{\Gamma(\mu h|y|) - 2} Q^*_{\Gamma(\mu h|y|) }\\
& \leq C \|\varphi\|_{\mathcal{B}^{M,h}(\Theta)}  e^{-\omega_{N^*}\left (\frac{1}{A h |y|} \right)}\end{align*}
and
\begin{align*}
|S_3(x,y)| &\leq C \sum_{\Gamma(\mu h|y|) < 2p +j \leq \Gamma(\mu h|y|/2)} \frac{|y|^{2p+j}}{(2p+j)!} |\Delta^{p} \varphi(x)|  (q^*_{2p+j})^2 |\chi''( \mu h q^*_{2p+j} y)| \\
& \leq C \|\varphi\|_{\mathcal{B}^{M,h}(\Theta)} \sum_{\Gamma(\mu h|y|) < 2p+j \leq \Gamma(\mu h|y|/2)} \frac{|y|^{2p+j-2}}{(2p+j-2)!} (\sqrt{d}h)^{2p}M_{2p} \\
& \leq C \|\varphi\|_{\mathcal{B}^{M,h}(\Theta)}  e^{-\omega_{N^*}\left (\frac{1}{A h |y|} \right)}.
\end{align*}
Next, we show \eqref{tech-1} and \eqref{tech-2}. We only treat the case $j = 0$ as the case $j =1$ is similar.  Let $\alpha \in \N^d$, $|\alpha| \leq 2$, be arbitrary. For all $(x,y) \in \Theta \times \R\backslash\{0\}$ with $|y|$ small enough it holds that
\begin{align*}
 |\partial^\alpha_x \Phi_0(x,y) - \partial^\alpha \varphi(x) | & \leq  \sum_{1 \leq p \leq \Gamma(\mu h|y|/2)/2}  \frac{|y|^{2p}}{(2p)!} |\Delta^p \partial^\alpha \varphi(x)| |\chi(\mu h q^*_{2p} y)| \\
& \leq C \|\varphi\|_{\mathcal{B}^{M,h}(\Theta)}  \sum_{1 \leq p \leq \Gamma(\mu h|y|/2)/2}  \frac{|y|^{2p}}{(2p)!} (\sqrt{d}h)^{2p}M_{2p+2} \\
& \leq C \|\varphi\|_{\mathcal{B}^{M,h}(\Theta)}   \sum_{1 \leq p \leq \Gamma(\mu h|y|/2)/2} \frac{1}{2^p} (\mu h|y|/2)^{2p} Q^*_{2p} \\
& \leq  |y|^2 C \|\varphi\|_{\mathcal{B}^{M,h}(\Theta)}.  
\end{align*}
Similarly, for all $\alpha \in \N^d$, $|\alpha| \leq 1$, and $(x,y) \in \Theta \times \R$ with $|y|$ small enough, we have 
\begin{align*}
| \partial_y\partial^\alpha_x \Phi_0(x,y)| &\leq \sum_{1 \leq p \leq \Gamma(\mu h|y|/2)/2}  \Big( \frac{|y|^{2p-1}}{(2p-1)!} | \Delta^p\partial^\alpha\varphi(x) | |\chi(\mu h q^*_{2p} y)| \\
&+ \frac{|y|^{2p}}{(2p)!} | \Delta^p\partial^\alpha\varphi(x)|\mu hq^*_{2p} |\chi'(\mu h q^*_{2p} y)| \Big)\\
&\leq  C \|\varphi\|_{\mathcal{B}^{M,h}(\Theta)}  \sum_{1 \leq p \leq \Gamma(\mu h|y|/2)/2}  \frac{|y|^{2p-1}}{(2p-1)!}  (\sqrt{d}h)^{2p}M_{2p+1} \\
& \leq C \|\varphi\|_{\mathcal{B}^{M,h}(\Theta)}   \sum_{1 \leq p \leq \Gamma(\mu h|y|/2)/2} \frac{1}{2^p} (\mu h|y|/2)^{2p-1} Q^*_{2p-1} \\
& \leq  |y| C \|\varphi\|_{\mathcal{B}^{M,h}(\Theta)}.  
\end{align*}
Finally, \eqref{sobolev-bounds} follows from an inspection of the estimates in the proofs of \eqref{tech-0}-\eqref{tech-2}.
\end{proof}
\begin{proposition}\label{almost-harmonic-compact-1} 
Let $M$, $N$, and $Q$ be three non-quasianalytic weight sequences satisfying \eqref{M2'double}. Suppose that $Q$ satisfies $(M.1)^{*}$. There is $A>0$ such that for all  $K \subsetc \R^d$ and $\varepsilon, h > 0$ the following holds: For all $\varphi_0, \varphi_1 \in \mathcal{D}^{M,h}_K$ there exists $\Phi = \Phi(\varphi_0, \varphi_1) \in C^2(\R^{d+1})$ with $\operatorname{supp} \Phi \subseteq K \times [-\varepsilon, \varepsilon]$ such that 
\begin{itemize}
\item[$(i)$] $\Phi_{| \Theta} = \varphi_0$ and $\partial_y \Phi_{| \Theta} = \varphi_1$;
\item[$(ii)$] $\displaystyle ||| \Phi ||| = \sup_{(x,y) \in \R^{d+1}}  |\Delta\Phi(x,y)| e^{\omega_{N^*}\left(\frac{1}{Ah|y|}\right)} < \infty$.
\end{itemize}
Moreover, there is $C > 0$ such that for all $\varphi_0,\varphi_1 \in \mathcal{D}^{M,h}_K$
$$
\max\{ ||| \Phi(\varphi_0, \varphi_1) |||, \max_{\substack{ \alpha \in \N^{d+1}; \\ |\alpha| \leq 1}} \| \partial^{\alpha}\Phi(\varphi_0, \varphi_1)\|_{L^\infty(\R^{d+1})} \} \leq C\max_{j = 0,1} \{ \| \varphi_j\|_{\mathcal{B}^{M,h}(\R^d)} \}.
$$
\end{proposition}
\begin{proof}
Choose $\psi \in \mathcal{D}(\R)$ such that $\operatorname{supp} \psi \subseteq [-\varepsilon,\varepsilon]$ and  $\psi \equiv 1$ on a neighborhood of $0$. Let $\Phi = \Phi(\varphi_0, \varphi_1)$ be the function from Proposition \ref{almost-harmonic} but call it $\widetilde{\Phi}$ instead of $\Phi$. Set $\Phi(x,y) = \psi(y) \widetilde{\Phi}(x,y)$. The definition of $\widetilde{\Phi}$ (see \eqref{formula-P} and \eqref{formula-P-1}) implies that $\operatorname{supp} \Phi \subseteq K \times [-\varepsilon, \varepsilon]$.  Since  $\widetilde{\Phi}$ satisfies the conditions of Proposition  for $\Omega = \R^d$, $\Phi$ satisfies all requirements.
\end{proof}
\begin{proposition}\label{sufficient-ah}
Let $M$ be a weight sequence satisfying $(NA)$. There is $A>0$ such that for all $V \subseteq \R^{d+1}$ open and $h > 0$ the following holds: Let $\Phi \in C^2(V)$ be such that
\begin{equation}
 \sup_{(x,y) \in V} |\Delta\Phi(x,y)| e^{\omega_{M^*}\left(\frac{1}{h|y|}\right)} < \infty.
\label{bound-almost-harmonic}
\end{equation}
Then, both $\Phi_{| \Theta}$ and $\partial_y \Phi_{|\Theta}$ belong to $\mathcal{B}^{M,Ah}(\Theta)$ for all $\Theta \Subset V \cap \R^d$.
\end{proposition}
We need some preparation for the proof of Proposition \ref{sufficient-ah}. Consider the following fundamental solution of the Laplacian
$$
E(x,y) =  \frac{1}{2\pi} \log |(x,y)|, \qquad (x,y) \in \R^2 \backslash \{0\},
$$
and for $d > 1$
$$
E(x,y) = \frac{-1}{(d-1)c_{d+1}|(x,y)|^{d-1}},  \qquad (x,y) \in \R^{d+1} \backslash \{0\},
$$
where $c_{d+1}$ denotes the area of the unit sphere in $\R^{d+1}$. The Poisson kernel is given by
$$
P(x,y) = \partial_yE(x,y) =  \frac{y}{c_{d+1}|(x,y)|^{d+1}}, \qquad (x,y) \in \R^{d+1} \backslash \{0\}.
$$
We need the following bounds for the derivatives of $E$ and $P$.
\begin{lemma}\label{bounds} \mbox{}
\begin{itemize} 
\item[$(i)$] There are $C, H > 0$ such that
$$
 | \partial^n_x E(x,y)| \leq   \frac{CH^{n}n!\max\{1,|\log |(x,y)||\}}{ |(x,y)|^n}, \qquad n \in \N, (x,y) \in \R^{2} \backslash \{0\},
$$
and for $d > 1$
$$
 | \partial^\alpha_x E(x,y)| \leq  \frac{CH^{|\alpha|}|\alpha|!}{|(x,y)|^{|\alpha|+d-1}}, \qquad \alpha \in \N^d, (x,y) \in \R^{d+1} \backslash \{0\}.
$$
\item[$(ii)$] There are $C,H > 0$ such that
$$
 | \partial^\alpha_x P(x,y)| \leq  \frac{CH^{|\alpha|}|\alpha|!}{|(x,y)|^{|\alpha|+d}}, \qquad \alpha \in \N^d, (x,y) \in \R^{d+1} \backslash \{0\}.
$$
\end{itemize}
\end{lemma}
\begin{proof}
We only show $(ii)$ as $(i)$ can be treated similarly. We will use the following property of harmonic functions (cf.\ \cite[p.\ 29, Thm.\ 7]{Evans}): There are $C, H > 0$  such that for all $w \in \R^{d+1}$ and $r > 0$
$$
|\partial^\beta U({w})| \leq \frac{CH^{|\beta|}|\beta|!}{r^{|\beta|}} \| U \|_{L^\infty(B(w,r))}, \qquad   \beta \in \N^{d+1},
$$
for all functions $U$ that are harmonic in a neighborhood of $\overline{B}(w,r)$. Fix $(x,y) \in \R^{d+1} \backslash \{0\}$ and $\alpha \in \N^d$. By applying the above inequality to $w = (x,y)$, $r = |(x,y)|/2$, $\beta = (\alpha,0)$ and $U = P$, we obtain
$$
 | \partial^\alpha_x P(x,y)| \leq \frac{C(2H)^{|\alpha|} |\alpha|!}{|(x,y)|^{|\alpha|}} \| P \|_{L^\infty(B((x,y),|(x,y)|/2))} .
$$
The result now follows from the inequality
$$
\| P \|_{L^\infty(B((x,y),|(x,y)|/2))} \leq \frac{2^d}{c_{d+1}|(x,y)|^{d}}.
$$
\end{proof}
\begin{proof}[Proof of Proposition \ref{sufficient-ah}] We only show the statement for $\partial_y \Phi$ as the one for $\Phi$ can be shown similarly. Set $V \cap \R^d = \Omega$ and $\varphi = \partial_y \Phi_{|\Omega}$. Fix an arbitrary $x_0 \in \Omega$ and choose $r > 0$ such that $B^{d+1}(x_0,r) \Subset V$. Note that
$$
\Phi(x,y) = \int_{B^{d+1}(x_0,r)} E(x-\xi,y-\eta) \Delta \Phi(\xi,\eta) \dxi\deta + U(x,y), \qquad (x,y) \in B^{d+1}(x_0,r),
$$ 
for some $U \in \mathcal{H}(B^{d+1}(x_0,r))$. Hence,
$$
\varphi(x) = - \int_{B^{d+1}(x_0,r)} P(x-\xi,\eta) \Delta \Phi(\xi,\eta) \dxi\deta + \partial_yU(x,0), \qquad x \in  B^{d}(x_0,r).
$$
Since $U \in \mathcal{H}(B^{d+1}(x_0,r))$, we have that  $\partial_y U_{|  B^{d}(x_0,r)} \in \mathcal{A}(B^{d}(x_0,r)) \subset \mathcal{E}^{(M)}(B^d(x_0,r))$. Set
$$
\psi(x) =  - \int_{B^{d+1}(x_0,r)} P(x-\xi,\eta) \Delta \Phi(\xi,\eta) \dxi\deta, \qquad x \in B^d(x_0,r).
$$
Lemma \ref{bounds}$(ii)$ and \eqref{bound-almost-harmonic} yield that $\psi \in C^\infty(B^d(x_0,r))$ with
$$
\partial^\alpha\psi(x) =   - \int_{B^{d+1}(x_0,r)} \partial^\alpha_xP(x-\xi,\eta) \Delta \Phi(\xi,\eta) \dxi\deta, \qquad  \alpha \in \N^d,
$$
and that there are $C,H > 0$ such that 
\begin{align*}
|\partial^\alpha\psi(x)| &\leq CH^{|\alpha|} |\alpha|! \int_{B^{d+1}(x_0,r)} \frac{e^{-\omega_{M^*}\left(\frac{1}{h|\eta|}\right)}}{(|x-\xi|^2 + \eta^2)^{(|\alpha|+d)/2}}  \dxi\deta  \\
&\leq C(Hh)^{|\alpha|} M_{|\alpha|} \int_{B^{d+1}(x_0,r)} \frac{1}{(|x-\xi|^2 + \eta^2)^{d/2}}  \dxi\deta \\
&\leq C(Hh)^{|\alpha|} M_{|\alpha|} \int_{B^{d+1}(0,2r)} \frac{1}{|(\xi,\eta)|^d}  \dxi\deta 
 \end{align*}
for all $\alpha \in \N^d$ and  $x \in B^d(x_0,r)$. Since $x_0$ was arbitrary, this proves the result.
\end{proof}
\begin{proof}[Proofs of Theorem \ref{main-part-1} and Theorem \ref{main-part-2}] By Lemma \ref{equivalent}, we may assume that each $M \in \mathfrak{M}$ satisfies $(M.1)^{*}$. Hence, the direct implication in Theorem \ref{main-part-1} (Theorem \ref{main-part-2}, repectively) follows from Proposition \ref{almost-harmonic} (Proposition \ref{almost-harmonic-compact-1}, respectively), while the reverse ones follow from Proposition \ref{sufficient-ah}.
\end{proof}

Theorem \ref{main-part-1} particularly applies to $\mathcal{E}^{[M]}(\Omega)$, where $M$ is a weight sequence satisfying $(M.2)'$, $(M.1)^{*}_{\mathrm{w}}$, and $(NA)$. Similarly,  Theorem \ref{main-part-2} applies to $\mathcal{D}^{[M]}(\Omega)$, where $M$ is a non-quasianalytic weight sequence satisfying $(M.2)'$ and $(M.1)^{*}_{\mathrm{w}}$. Furthermore, our results yield characterizations of  $\mathcal{E}^{[\omega]}(\Omega)$ and $\mathcal{D}^{[\omega]}(\Omega)$ by almost harmonic functions as well:

\begin{corollary}
Let $\omega$ be a weight function satisfying $(\alpha_0)$ and $\omega(t) = o(t)$. Let $\Omega \subseteq \R^d$ be open and let $\varphi_0, \varphi _1 : \Omega \rightarrow \C$. Then, $\varphi_0, \varphi_1 \in \mathcal{E}^{[\omega]}(\Omega)$ if and only if for all $\Theta \Subset \Omega$ and for all $h > 0$ (for some $h > 0$) the following holds: For some/all $V \subset \R^{d+1}$ open with $V \cap \R^d = \Theta$ there exists $\Phi \in C^2(V)$ such that $\Phi_{|\Theta} = \varphi_{0|\Theta}$, $\partial_y \Phi_{|\Theta} = \varphi_{1|\Theta}$, and
$$
 \sup_{(x,y) \in V} |\Delta\Phi(x,y)| e^{\frac{1}{h} \omega^\star(h|y|)}< \infty.
 $$
\end{corollary}
\begin{proof}
This follows from Lemma \ref{reduction}, Lemma \ref{reduction-1}, and Theorem \ref{main-part-1}.
\end{proof}
\begin{corollary}
Let $\omega$ be a non-quasianalytic weight function satisfying $(\alpha_0)$. Let $\Omega \subseteq \R^d$ be open and let $V \subseteq \R^{d+1}$ be open such that $V \cap \R^d = \Omega$. Let $\varphi_0,\varphi_1: \Omega \rightarrow \C$. Then, $\varphi_0,\varphi_1 \in \mathcal{D}^{[\omega]}(\Omega)$ if and only if for all $h > 0$ (for some $h > 0$) there exists $\Phi \in C^2_c(V)$ such that $\Phi_{|\Omega} = \varphi_{0}$, $\partial_y \Phi_{|\Omega} = \varphi_{1}$, and
$$
 \sup_{(x,y) \in V} |\Delta\Phi(x,y)|  e^{\frac{1}{h} \omega^\star(h|y|)}< \infty.
 $$
\end{corollary}
\begin{proof}
This follows from Lemma \ref{reduction}, Lemma \ref{reduction-1},  and Theorem \ref{main-part-2}.
\end{proof}
\section{Boundary values of harmonic functions}\label{sect-bv}

\subsection{Analytic functionals via harmonic functions} Given  $\Omega \subseteq \R^d$ open, we denote by  $\mathcal{A}^{'}(\Omega)$ the dual of $\mathcal{A}(\Omega)$. Since the space of entire functions is dense in $\mathcal{A}(\Omega)$, we may view  $\mathcal{A}^{'}(\Omega)$
as a subspace of $\mathcal{A}^{'}(\R^d)$. A compact set $K$ in $\R^d$ is said to be a \emph{carrier} of $f \in \mathcal{A}^{'}(\R^d)$ if $f \in \mathcal{A}^{'}(\Omega)$ for all $\Omega \subseteq \R^d$ open with $K \subsetc \Omega$.  We denote by $\mathcal{A}^{'}(K)$ the space consisting of all $f \in  \mathcal{A}^{'}(\R^d)$ such that $K$ is a carrier of $f$. The space  $\mathcal{A}^{'}(K)$ may be characterized in terms of harmonic functions, as we now proceed to explain. We follow H\"ormander's exposition \cite[Section 9.1]{Hormander}  (see also \cite{Komatsu1991, Komatsu2, Schapira}).

Let $K \subsetc \R^d$. The Poisson transform of $f \in \mathcal{A}^{'}(K)$ is defined as 
$$
P[f](x,y) := \langle f(\xi), P(x-\xi,y) \rangle, \qquad (x,y) \in \R^{d+1}\backslash K. 
$$
Recall that $\mathcal{H}_-(\R^{d+1} \backslash K)$ stands for the space of harmonic functions in $\R^{d+1} \backslash K$ that are odd with respect to $y$. We denote by $\mathcal{H}_{0,-}(\R^{d+1} \backslash K)$ the space consisting of all $F  \in \mathcal{H}_{-}(\R^{d+1} \backslash K)$ such that
$F(x,y) \rightarrow 0$ as $(x,y) \to \infty$.
\begin{theorem} \label{Hormander-forever}  Let $K \subsetc \R^d$.
\begin{itemize}
\item[$(i)$] \cite[Proposition 9.1.3]{Hormander} Let $f \in \mathcal{A}^{'}(K)$. Then, $P[f] \in \mathcal{H}_{0,-}(\R^{d+1} \backslash K)$ and
$$
\langle f, \partial_{y} \Phi_{| \R^d} \rangle =  - \int_{\R^{d+1}} P[f](x,y) \Delta(\rho\Phi)(x,y) \dx\dy,
$$
for all $\Phi \in \mathcal{H}(\R^{d+1})$ and $\rho \in \mathcal{D}(\R^{d+1})$ such that $\rho \equiv 1$ on an $\R^{d+1}$-neighborhood of $K$.
\item[$(ii)$] \cite[Proposition 9.1.5]{Hormander} Let $F  \in \mathcal{H}_-(\R^{d+1} \backslash K)$. Then, there exists a unique $f \in \mathcal{A}^{'}(K)$ such that
$$
\langle f, \partial_{y} \Phi_{| \R^d} \rangle =  - \int_{\R^{d+1}} F(x,y) \Delta(\rho\Phi)(x,y) \dx\dy,
$$
for all $\Phi \in \mathcal{H}(\R^{d+1})$ and $\rho \in \mathcal{D}(\R^{d+1})$ such that $\rho \equiv 1$ on an $\R^{d+1}$-neighborhood of $K$. Moreover, there is $U \in \mathcal{H}(\R^{d+1})$ such that $F = P[f] + U$. 
\end{itemize}
\end{theorem}

\begin{corollary}\label{Hormander-cor}
 Let $K \subsetc \R^d$ and let $V$ be an open $\R^{d+1}$-neighborhood of $K$ that is symmetric with respect to $y$.  For each $F  \in \mathcal{H}_-(V \backslash K)$  there is a unique $f \in \mathcal{A}^{'}(K)$ such that
\begin{equation}
\langle f, \partial_{y} \Phi_{| \R^d} \rangle =  - \int_{\R^{d+1}} F(x,y) \Delta(\rho\Phi)(x,y) \dx\dy,
\label{green}
\end{equation}
for all $\Phi \in \mathcal{H}(\R^{d+1})$ and $\rho \in \mathcal{D}(V)$ such that $\rho \equiv 1$ on an $\R^{d+1}$-neighborhood of $K$. Moreover, there is $U \in \mathcal{H}(V)$ such that $F = P[f] + U$. 
\end{corollary}
\begin{proof}
Since $V \backslash K = (\R^{d+1} \backslash K) \cap V$, the Mittag-Leffler theorem for harmonic functions \cite[Theorem 2.4]{Komatsu1991}
 implies that there are $F_1  \in \mathcal{H}_-(\R^{d+1} \backslash K)$ and $F_2 \in \mathcal{H}_-(V)$ such that $F = F_1 - F_2$ on $V \backslash K$. The result therefore follows from  Theorem \ref{Hormander-forever}$(ii)$ (applied to $F_1$).
\end{proof}

\begin{remark} \label{remark-bv}
Let $K \subsetc \R^d$ and let $V$ be an open $\R^{d+1}$-neighborhood of $K$ that is symmetric with respect to $y$. Let  $F  \in \mathcal{H}_{-}(V \backslash K)$ and consider the associated $f \in \mathcal{A}^{'}(K)$ from Corollary \ref{Hormander-cor}. By Green's theorem (cf.\ the proof of Proposition \ref{bv-theorem} below), we have 
\begin{align*}
\langle f, \partial_{y} \Phi_{| \R^d} \rangle &=  - \int_{\R^{d+1}} F(x,y) \Delta(\rho\Phi)(x,y) \dx\dy \\
&= \lim_{y \to 0^+} \int_{\R^d} (F(x,y) - F(x,-y)) \rho(x,0) \partial_{y} \Phi(x,0) \dx 
\end{align*}
for all $\Phi \in \mathcal{H}(\R^{d+1})$ and $\rho \in \mathcal{D}(V)$ even with respect to $y$ such that $\rho \equiv 1$ on an $\R^{d+1}$-neighborhood of $K$. Hence, $f$ may be interpreted as the boundary value of $F$ in  $\mathcal{A}^{'}(K)$ and we write $f = \operatorname{bv}(F)$.
\end{remark}

\subsection{Spaces of ultradistributions} \label{sect-ud}
Let $\mathfrak{M}$ be a weight matrix satisfying $[\mathfrak{M}.2]'$ and $(NA)$. Given  $\Omega \subseteq \R^d$ open, we denote by  $\mathcal{E}^{' [\mathfrak{M}]}(\Omega)$ the strong dual of $\mathcal{E}^{[\mathfrak{M}]}(\Omega)$. We have once again that the space of entire functions is dense in $\mathcal{E}^{ [\mathfrak{M}]}(\Omega)$ (cf.\ \cite[Proposition 3.2]{Hormander85}), we therefore obtain that  $\mathcal{E}^{' [\mathfrak{M}]}(\Omega)$ may be viewed as a subspace of $\mathcal{A}^{'}(\R^d)$.  A compact set $K$ in $\R^d$ is said to be an \emph{$[\mathfrak{M}]$-carrier} of $f \in \mathcal{E}^{' [\mathfrak{M}]}(\R^d)$ if $f \in \mathcal{E}^{' [\mathfrak{M}]}(\Omega)$ for all $\Omega \subseteq \R^d$ open with $K \subsetc \Omega$. We denote by $\mathcal{E}^{' [\mathfrak{M}]}(K)$ the space consisting of all $f \in \mathcal{E}^{' [\mathfrak{M}]}(\R^d)$ such that $K$ is an $[\mathfrak{M}]$-carrier of $f$.
We have the following canonical isomorphism of vector spaces
$$
\mathcal{E}^{' [\mathfrak{M}]}(K) \cong \varprojlim_{K \subsetc \Omega} \mathcal{E}^{' [\mathfrak{M}]}(\Omega).
$$
We endow $\mathcal{E}^{' [\mathfrak{M}]}(K)$ with the projective limit topology induced by this isomorphism. 

Suppose that $\mathfrak{M}$ is non-quasianalytic. Given  $\Omega \subseteq \R^d$ open, we denote by  $\mathcal{D}^{' [\mathfrak{M}]}(\Omega)$ the strong dual of $\mathcal{D}^{[\mathfrak{M}]}(\Omega)$.

\subsection{Boundary values of harmonic functions in $\mathcal{E}^{'[\mathfrak{M}]}(K)$} 
Let $M$ be a weight sequence satisfying $(NA)$. Let  $V \subseteq \R^{d+1}$ be open and symmetric with respect to $y$ and let $S \subseteq V \cap \R^d$ be closed in $V$. For $h > 0$ we write $\mathcal{H}^{M,h}_{\infty,-}(V\backslash S)$ for the Banach space consisting of all $F \in \mathcal{H}_-(V\backslash S)$ such that 
$$
\| F \|_{\mathcal{H}^{M,h}_{\infty,-}(V\backslash S)} := \sup_{(x,y) \in V \backslash S} |F(x,y)| e^{-\omega_{M^*}\left(\frac{1}{hd_S(x,y)}\right)} < \infty,
$$
where $d_S(x,y)$ denotes the distance from $(x,y)$ to $S$. We set
$$
\mathcal{H}^{(M)}_{\infty,-}(V\backslash S) :=  \varinjlim_{h \to 0^+} \mathcal{H}^{M,h}_{\infty,-}(V\backslash S), \qquad
\mathcal{H}^{\{M\}}_{\infty,-}(V\backslash S) :=  \varprojlim_{h \to \infty} \mathcal{H}^{M,h}_{\infty,-}(V\backslash S).
$$
Next, let $\mathfrak{M}$ be a weight matrix satisfying $(NA)$. Let $K \subsetc \R^d$ and let $V$ be an open $\R^{d+1}$-neighborhood of $K$ that is symmetric with respect to $y$.  Choose a sequence $(V_n)_{n \in \N}$ of relatively compact open sets in $\R^{d+1}$ that are symmetric with respect to $y$ such that $K \subsetc V_n$, $V_n \Subset V_{n+1}$ and $V = \bigcup_{n \in \N} V_n$, and a sequence $(K_n)_{n \in \N}$ of  compact sets in $\R^{d}$ such that $K \subsetc \operatorname{int} K_n$, $K_{n+1} \subsetc \operatorname{int} K_{n}$, $K_n \subsetc V_n$ and $K = \bigcap_{n \in \N} K_n$. We define
\begin{gather*}
\mathcal{H}^{(\mathfrak{M})}_-(V\backslash K) := \varprojlim_{n \in \N} \varinjlim_{M \in \mathfrak{M}}\mathcal{H}^{(M)}_{\infty,-}(V_n\backslash K_n) , \qquad
\mathcal{H}^{\{\mathfrak{M}\}}_-(V\backslash K) := \varprojlim_{n \in \N} \varprojlim_{M \in \mathfrak{M}} \mathcal{H}^{\{M\}}_{\infty,-}(V_n\backslash K_n) .
\end{gather*}
This definition is independent of the chosen sequences $(V_n)_{n \in \N}$  and $(K_n)_{n \in \N}$. 
For two weight matrices $\mathfrak{M}$  and $\mathfrak{N}$ with $\mathfrak{M} [\approx] \mathfrak{N}$ we have that $\mathcal{H}^{[\mathfrak{M}]}_-(V\backslash K) = \mathcal{H}^{[\mathfrak{N}]}_-(V\backslash K)$ as locally convex spaces.

Let $K \subsetc \R^d$ and let $V$ be an open $\R^{d+1}$-neighborhood of $K$ that is symmetric with respect to $y$. Recall from Remark \ref{remark-bv} that we employ the notation $f=\operatorname{bv}(F)$ for the analytic functional corresponding to a harmonic function $F\in\mathcal{H}_{-}(V\backslash K)$ via the relation \eqref{green}.  We now show that the elements of $\mathcal{H}^{[\mathfrak{M}]}_-(V\backslash K)$ have boundary values in $\mathcal{E}^{'[\mathfrak{M}]}(K)$.

\begin{proposition}\label{bv-theorem}
Let $\mathfrak{M}$ be a weight matrix  satisfying $[\mathfrak{M}.1]^{*}_{\mathrm{w}}$, $[\mathfrak{M}.2]'$, and $(NA)$. Let $K \subsetc \R^d$ and let $V$ be an open $\R^{d+1}$-neighborhood of $K$ that is symmetric with respect to $y$. For  each $F \in \mathcal{H}^{[\mathfrak{M}]}_{-}(V \backslash K)$ we have that $\operatorname{bv}(F)  \in  \mathcal{E}^{'[\mathfrak{M}]}(K)$ and this quasianalytic functional may be represented as follows: For all  $\Omega \subseteq V \cap \R^d$ open with $K \subsetc \Omega$ it holds that
$$
\langle \operatorname{bv}(F), \varphi \rangle = \lim_{y \to 0^+} \int_{\R^d} (F(x,y) - F(x,-y)) \chi(x) \varphi(x) \dx, \qquad  \varphi \in \mathcal{E}^{[\mathfrak{M}]}(\Omega),
$$
 where $\chi \in \mathcal{D}(\Omega)$ is such that $\chi \equiv 1$ on a neighborhood of $K$.  

Moreover, the boundary value mapping
$$
\operatorname{bv}: \mathcal{H}^{[\mathfrak{M}]}_{-}(V \backslash K) \rightarrow  \mathcal{E}^{'[\mathfrak{M}]}(K)
$$
is continuous.
\end{proposition}

As stated in the introduction, we shall show Proposition \ref{bv-theorem} by combining Green's theorem with our description of ultradifferentiable functions by almost harmonic functions (Proposition \ref{almost-harmonic}). This method is  suggested by \eqref{green} (see  Remark  \ref{remark-bv}).

\begin{proof}[Proof of Proposition \ref{bv-theorem}]
We only consider the Beurling case as the Roumieu case can be treated similarly. By Lemma \ref{equivalent}, we may assume that each $M \in \mathfrak{M}$ satisfies $(M.1)^{*}$. 

Fix an arbitrary open subset $\Omega \subseteq V \cap \R^d$ with $K \subsetc \Omega$ and let $\chi \in \mathcal{D}(\Omega)$ be such that $\chi \equiv 1$ on a neighborhood of $K$. 
Choose $\Theta \Subset \Omega$ with piecewise smooth boundary such that  $\supp \chi \subsetc \Theta$.  Let $r > 0$ be such that $\Theta \times (-r,r)  \Subset V$. Pick $L \subsetc \Theta$  such that $K \subsetc \operatorname{int} L$ and $\chi \equiv 1$ on a neighborhood of $L$.  It suffices to show that for all $N \in \mathfrak{M}$ and $k >0$ there is $M \in \mathfrak{M}$ such that 
\begin{equation}
\operatorname{bv}_\Omega: \mathcal{H}^{N,k}_{\infty,-}( \Theta  \times (-r,r) \backslash L) \rightarrow \mathcal{E}^{'(M)}(\Omega),
\label{inter-bv}
\end{equation}
where
 \begin{align*}
\langle \operatorname{bv}_\Omega(F) , \varphi \rangle &= \lim_{y \to 0^+} \int_{\R^d} (F(x,y) - F(x,-y)) \chi(x) \varphi(x) \dx \\ 
&= 2 \lim_{y \to 0^+} \int_{\R^d} F(x,y) \chi(x) \varphi(x) \dx, \qquad \varphi \in \mathcal{E}^{(M)}(\Omega),
\end{align*}
is well-defined and continuous. Choose $Q,M \in \mathfrak{M}$ such that \eqref{M2'double} holds. For $\varphi \in \mathcal{E}^{(M)}(\Omega)$ consider the function $\Phi = \Phi(0,\varphi_{|\Theta}) \in C^2(\Theta \times \R)$ from Proposition \ref{almost-harmonic} with $h = k/A$. Let $\varepsilon > 0$ be such that $\Theta \times (-r-\varepsilon,r+\varepsilon)  \Subset V$.
For  $F \in  \mathcal{H}^{N,k}_{\infty,-}( \Theta  \times (-r,r) \backslash L)$ and $0 < y < \varepsilon$ we set  $F_y(x,\eta) = F(x, \eta + y)$. Then, $F_y$ is harmonic in a neighborhood of $\overline{\Theta} \times [0,r]$. Choose $\rho \in \mathcal{D}(\Theta \times (-r,r))$ such that $\rho \equiv 1$ on an $\R^{d+1}$-neighborhood of $L$ and $\rho_{|\R^d} = \chi$. By applying Green's theorem to the pair $(F_y, \rho \Phi)$ on the region $\Theta \times (0,r)$, we obtain 
$$
\int_{\R^d} F(x,y) \chi(x) \varphi(x) \dx = - \int_{\Theta \times (0,r)} F(x, \eta + y) \Delta( \rho \Phi)(x,\eta) \dx \deta
$$
for all $0 < y < \varepsilon$. Let $J \subsetc \Theta \times (-r,r)$ be such that $L \subsetc \operatorname{int} J$ and $\rho \equiv 1$ on $J$.  Property $(ii)$ of Proposition \ref{almost-harmonic} implies that
$$
\lim_{y \to 0^+} \int_{\Theta \times (0,r) \cap J} F(x, \eta + y) \Delta (\rho \Phi) (x,\eta) \dx \deta = \int_{\Theta \times (0,r) \cap J} F(x, \eta) \Delta \Phi (x,\eta) \dx \deta 
$$
and
\begin{align*}
\left | \int_{\Theta \times (0,r) \cap J} F(x, \eta) \Delta \Phi (x,\eta) \dx \deta \right | 
\leq |\Theta \times (0,r) \cap J| \|F\| ||| \Phi |||, 
\end{align*}
 where $ \| F \| =  \|F\|_{\mathcal{H}^{N,k}_{\infty,-}(\Theta  \times (-r,r)\backslash L)}$ and $||| \Phi |||= \sup_{(x,\eta) \in \Theta \times \R}  |\Delta\Phi(x,\eta)| e^{\omega_{N^*}\left(\frac{1}{k|\eta|}\right)}$ . Since  $F$ is continuous on a neighborhood of $\overline{\Theta} \times [0,r] \backslash J$, we have 
$$
\lim_{y \to 0^+} \int_{\Theta \times (0,r) \backslash J} F(x, \eta + y) \Delta( \rho \Phi)(x,\eta) \dx \deta = \int_{\Theta \times (0,r) \backslash J} F(x, \eta) \Delta( \rho \Phi)(x,\eta) \dx \deta
$$
and, if $\delta > 0$ is such that $d_L(x,y) \geq \delta$ for all $(x,y) \in \Theta \times (0,r) \backslash J$, then
$$
\left |\int_{\Theta \times (0,r) \backslash J} F(x, \eta) \Delta( \rho \Phi)(x,\eta) \dx \deta \right | \leq e^{\omega_{N^*}\left(\frac{1}{k\delta}\right)} | \Theta \times (0,r) \backslash J|  \|F\| \|\Delta( \rho \Phi)\|_{L^\infty(\Theta \times \R)}.
$$
Note that there is $C > 0$ (independent of $\varphi$) such that 
$$
 \|\Delta( \rho \Phi)\|_{L^\infty(\Theta \times \R)} \leq C \max\{ ||| \Phi |||,  \max_{\substack{ \alpha \in \N^{d+1}; \\ |\alpha| \leq 1}} \| \partial^\alpha\Phi\|_{L^\infty(\Theta\times \R)}\}.
$$
Hence,
$$
\langle \operatorname{bv}_\Omega(F) , \varphi \rangle =  - 2\int_{\Theta \times (0,r)} F(x, \eta) \Delta( \rho \Phi)(x,\eta) \dx \deta, \qquad \varphi \in  \mathcal{E}^{(M)}(\Omega),
$$
and there is $C > 0$ such that
$$
|\langle \operatorname{bv}_\Omega(F) , \varphi \rangle| \leq C\|F\| \max\{ ||| \Phi(0,\varphi_{|\Theta}) |||, \max_{\substack{ \alpha \in \N^{d+1}; \\ |\alpha| \leq 1}} \| \partial^{\alpha}\Phi(0,\varphi_{|\Theta})\|_{L^\infty(\Theta\times \R)} \}
$$
for all  $F \in  \mathcal{H}^{N,k}_{\infty,-}(\Theta  \times (-r,r) \backslash L)$ and $\varphi \in \mathcal{E}^{(M)}(\Omega)$. Therefore,  \eqref{sobolev-bounds} implies that the mapping in \eqref{inter-bv} is well-defined and continuous.
\end{proof}

Our next goal is to study the Poisson transform of elements of $\mathcal{E}^{'[\mathfrak{M}]}(K)$. To this end, we need to introduce some additional spaces of harmonic functions.  Let $M$ be a weight sequence satisfying $(NA)$.  Let $K \subsetc \R^d$. Recall that $\mathcal{H}_{0,-}(\R^{d+1}\backslash K)$ stands for the space of harmonic functions on $\R^{d+1}\backslash K$ that are odd with respect to $y$ and vanish at infinity. For $h > 0$ we write $\mathcal{H}^{M,h}_{\infty,0,-}(\R^{d+1}\backslash K)$ for the Banach space consisting of all $F \in \mathcal{H}_{0,-}(\R^{d+1}\backslash K)$ such that 
$\| F \|_{\mathcal{H}^{M,h}_{\infty,-}(\R^{d+1}\backslash K)}< \infty $. We set
\begin{gather*}
\mathcal{H}^{(M)}_{\infty,0,-}(\R^{d+1}\backslash K) :=  \varinjlim_{h \to 0^+} \mathcal{H}^{M,h}_{\infty,0,-}(\R^{d+1}\backslash K), \\
\mathcal{H}^{\{M\}}_{\infty,0,-}(\R^{d+1}\backslash K) :=  \varprojlim_{h \to \infty} \mathcal{H}^{M,h}_{\infty,0,-}(\R^{d+1}\backslash K).
\end{gather*}
Next, let $\mathfrak{M}$ be a weight matrix satisfying $(NA)$. Let $K \subsetc \R^d$. Choose a sequence $(K_n)_{n \in \N}$ of  compact sets in $\R^{d}$ such that $K \subsetc \operatorname{int} K_n$, $K_{n+1} \subsetc \operatorname{int} K_{n}$ and $K = \bigcap_{n \in \N} K_n$. We define
\begin{gather*}
\mathcal{H}^{(\mathfrak{M})}_{0,-}(\R^{d+1} \backslash K) := \varprojlim_{n \in \N} \varinjlim_{M \in \mathfrak{M}}\mathcal{H}^{(M)}_{\infty,0,-}(\R^{d+1} \backslash K_n), \\
\mathcal{H}^{\{\mathfrak{M}\}}_{0,-}(\R^{d+1}\backslash K) := \varprojlim_{n \in \N} \varprojlim_{M \in \mathfrak{M}} \mathcal{H}^{\{M\}}_{\infty,0,-}(\R^{d+1}\backslash K_n) .
\end{gather*}
This definition is independent of the chosen sequence $(K_n)_{n \in \N}$. 

\begin{proposition}\label{poisson-prop}
Let $\mathfrak{M}$ be a weight matrix satisfying $[\mathfrak{M}.2]'$ and $(NA)$. Let $K \subsetc \R^d$. Then, the Poisson transform
$$
P[\, \cdot \, ]:  \mathcal{E}^{'[\mathfrak{M}]}(K) \rightarrow \mathcal{H}^{[\mathfrak{M}]}_{0,-}(\R^{d+1}\backslash K)
$$
is well-defined and continuous. 
\end{proposition}
\begin{proof}
We only consider the Beurling case as the Roumieu case is similar. It suffices to show that for all $\Omega \Subset \R^d$ with $K \subsetc \Omega$ and all $M \in \mathfrak{M}$ there is $N \in \mathfrak{M}$ such that
$$ 
P[\, \cdot \, ]:  \mathcal{E}^{'(M)}(\Omega) \rightarrow  \mathcal{H}^{(N)}_{\infty,0,-}(\R^{d+1}\backslash \overline{\Omega})
$$
is well-defined and continuous. Choose $N \in \mathfrak{M}$ such that
$$
N_{p+d} \leq C_0H_0^p M_p, \qquad p \in \N,
$$
for some $C_0,H_0 > 0$. By Theorem \ref{Hormander-forever}$(i)$, we have that $P[f] \in \mathcal{H}_{0,-}(\R^{d+1}\backslash \overline{\Omega})$ for all $f \in  \mathcal{E}^{'(M)}(\Omega) \subset \mathcal{A}^{'}(\overline{\Omega})$. Since  $\mathcal{E}^{'(M)}(\Omega)$ is bornological, it is enough to show that for every bounded set $B \subset \mathcal{E}^{'(M)}(\Omega)$  the set $\{ P[f] \, | \, f \in B \}$ is contained and bounded in $\mathcal{H}^{(N)}_{\infty,0,-}(\R^{d+1}\backslash \overline{\Omega})$. By the Banach-Steinhaus theorem, there are $C_1,h > 0$ and $\Theta \Subset \Omega$
such that
$$
|\langle f, \varphi \rangle | \leq C_1 \| \varphi \|_{\mathcal{B}^{M,h}(\Theta)}, \qquad \varphi \in  \mathcal{E}^{(M)}(\Omega), 
$$
for all $f \in B$. Lemma \ref{bounds}$(ii)$ therefore implies that for all $(x,y) \in \R^{d+1} \backslash \overline{\Omega}$
\begin{align*}
|P[f](x,y)| &= |\langle f(\xi), P(x-\xi,y) \rangle| \\
&\leq C_1 \sup_{\xi \in \Theta} \sup_{\alpha \in \N^d} \frac{|\partial^\alpha_\xi P(x-\xi,y)|}{h^{|\alpha|}M_{|\alpha|}} \\
&\leq CC_1 \sup_{\xi \in \Theta} \sup_{\alpha \in \N^d}\frac{H^{|\alpha|} |\alpha|!}{h^{|\alpha|}M_{|\alpha|}(|x-\xi|^2 + y^2)^{(|\alpha|+d)/2}} \\
&\leq CC_0C_1(h/HH_0)^{d} \sup_{\alpha \in \N^d}\frac{(HH_0)^{|\alpha|+d}(|\alpha|+d)!}{(hd_{\overline{\Omega}}(x,y))^{|\alpha|+d}N_{|\alpha|+d}}  \\
&\leq CC_0C_1(h/HH_0)^{d}  e^{\omega_{N^*}\left(\frac{HH_0}{hd_{\overline{\Omega}}(x,y)}\right)}
\end{align*}
for all $f \in B$.
\end{proof}
We are ready to prove the main result of this article.
\begin{theorem}\label{main-theorem} Let $\mathfrak{M}$ be a weight matrix satisfying $[\mathfrak{M}.1]^{*}_{\mathrm{w}}$, $[\mathfrak{M}.2]'$, and $(NA)$. Let $K \subsetc \R^d$.
\begin{itemize}
\item[$(i)$] Let $V$ be an open $\R^{d+1}$-neighborhood of $K$ that is symmetric with respect to $y$. Then, the sequence
$$
0 \longrightarrow \mathcal{H}_-(V) \longrightarrow  \mathcal{H}^{[\mathfrak{M}]}_-(V \backslash K) \xrightarrow{\phantom,\operatorname{bv}\phantom,}  \mathcal{E}^{'[\mathfrak{M}]}(K) \longrightarrow 0
$$
is  exact. Moreover, the boundary value mapping is continuous and it has  the Poisson transform
$$
P[\, \cdot \, ]: \mathcal{E}^{'[\mathfrak{M}]}(K)  \rightarrow \mathcal{H}^{[\mathfrak{M}]}_-(V \backslash K)
$$
 as a continuous linear right inverse.
 \item[$(ii)$] The boundary value mapping
$$
\operatorname{bv}: \mathcal{H}^{[\mathfrak{M}]}_{0,-}(\R^{d+1} \backslash K) \rightarrow  \mathcal{E}^{'[\mathfrak{M}]}(K)
$$
is a topological isomorphism whose inverse is given by the Poisson transform
$$
P[\, \cdot \, ]:  \mathcal{E}^{'[\mathfrak{M}]}(K) \rightarrow \mathcal{H}^{[\mathfrak{M}]}_{0,-}(\R^{d+1} \backslash K). 
$$ 
\end{itemize}
\end{theorem}
\begin{proof}
$(i)$ The boundary value mapping and the Poisson transform are well-defined and continuous by Proposition \ref{bv-theorem} and Proposition \ref{poisson-prop}, respectively. Theorem \ref{Hormander-forever}$(i)$ and Remark \ref{remark-bv} yield that $P$ is a right inverse of $\operatorname{bv}$. Finally, the equality  $\ker \operatorname{bv} = \mathcal{H}_-(V)$  follows from Corollary \ref{Hormander-cor} and Remark \ref{remark-bv}. 

$(ii)$ This follows from part $(i)$, Proposition \ref{poisson-prop}, and Liouville's theorem for harmonic functions.

\end{proof}

Theorem \ref{main-theorem} particularly applies to $\mathcal{E}^{'[M]}(K)$, where $M$ is a weight sequence satisfying $(M.2)'$, $(M.1)^{*}_{\mathrm{w}}$, and $(NA)$. Finally, we give two representations of $\mathcal{E}^{'[\omega]}(K)$ by boundary values of harmonic functions. Let $\omega$ be a weight function satisfying $\omega(t) = o(t)$. Let $V \subseteq \R^{d+1}$ be open and symmetric with respect to $y$ and let $S \subseteq V \cap \R^d$ be closed in $V$. For $h > 0$ we write $\mathcal{H}^{\omega,h}_{\infty,-}(V\backslash S)$ for the Banach space consisting of all $F \in \mathcal{H}_-(V\backslash S)$ such that 
$$
\| F \|_{\mathcal{H}^{\omega,h}_{\infty,-}(V\backslash S)} := \sup_{(x,y) \in V \backslash S} |F(x,y)| e^{-\frac{1}{h}\omega^\star(h d_S(x,y))} < \infty.
$$
 We set
$$
\mathcal{H}^{(\omega)}_{\infty,-}(V\backslash S) :=  \varinjlim_{h \to 0^+} \mathcal{H}^{\omega,h}_{\infty,-}(V\backslash S), \qquad
\mathcal{H}^{\{\omega\}}_{\infty,-}(V\backslash S) :=  \varprojlim_{h \to \infty} \mathcal{H}^{\omega,h}_{\infty,-}(V\backslash S).
$$
Let $K \subsetc \R^d$ and let $V$ be an open $\R^{d+1}$-neighborhood of $K$ that is symmetric with respect to $y$. The spaces $\mathcal{H}^{[\omega]}_-(V\backslash K)$ and $\mathcal{H}^{[\omega]}_{0,-}(\R^{d+1}\backslash K)$ are defined in the natural way.
\begin{corollary}
Let $\omega$ be a weight function satisfying $(\alpha_0)$ and $\omega(t) = o(t)$. Let $K \subsetc \R^d$.
\begin{itemize}
\item[$(i)$] Let $V$ be an open $\R^{d+1}$-neighborhood of $K$ that is symmetric with respect to $y$. Then, the sequence
$$
0 \longrightarrow \mathcal{H}_-(V) \longrightarrow  \mathcal{H}^{[\omega]}_-(V \backslash K) \xrightarrow{\phantom,\operatorname{bv}\phantom,}  \mathcal{E}^{'[\omega]}(K) \longrightarrow 0
$$
is  exact. Moreover, the boundary value mapping is continuous and it has  the Poisson transform
$$
P[\, \cdot \, ]: \mathcal{E}^{'[\omega]}(K)  \rightarrow \mathcal{H}^{[\omega]}_-(V \backslash K)
$$
as  a continuous linear right inverse.
 \item[$(ii)$] The boundary value mapping
$$
\operatorname{bv}: \mathcal{H}^{[\omega]}_{0,-}(\R^{d+1} \backslash K) \rightarrow  \mathcal{E}^{'[\omega]}(K)
$$
is a topological isomorphism whose inverse is given by the Poisson transform
$$
P[\, \cdot \, ]:  \mathcal{E}^{'[\omega]}(K) \rightarrow \mathcal{H}^{[\omega]}_{0,-}(\R^{d+1} \backslash K). 
$$ 
\end{itemize}
\end{corollary}
\begin{proof}
This follows from  Lemma \ref{reduction},  Lemma \ref{reduction-1},  and Theorem \ref{main-theorem}.
\end{proof}

\subsection{Application: The support theorem for quasianalytic functionals}
A fundamental result in the theory of analytic functionals states that each $f \in \mathcal{A}^{'}(\R^d)$ has a unique minimal carrier, called the \emph{support} of $f$ and denoted by $\operatorname{supp}_{\mathcal{A}^{'}}f$. Martineau \cite{Martineau} (see also \cite{Morimoto}) showed this by using  cohomological properties of the sheaf of germs of analytic functions. 

Theorem \ref{Hormander-forever} may be used to give a simpler proof of the existence of  $\operatorname{supp}_{\mathcal{A}^{'}}f$ (cf.\ \cite[Theorem 9.1.6]{Hormander}). In fact, by Theorem \ref{Hormander-forever}, a compact set $K$ in $\R^d$ is a  carrier of $f$ if and only if its Poisson transform $P[f]$ can be continued as a harmonic function to $\R^{d+1} \backslash K$. Hence,   $\operatorname{supp}_{\mathcal{A}^{'}}f$ is given by the compact set $K\subset \R^{d}$ with the property that $\R^{d+1}\backslash K$ is the largest open set in $\R^{d+1}$ on which $P[f]$ has  a harmonic extension and, in particular, this notion is well-defined.

The existence of a unique minimal carrier can also be established for quasianalytic functionals, but the only known treatment in the literature, due to H\"ormander \cite{Hormander85}, turns out to be much harder. Howeover, in view of Theorem \ref{main-theorem}, we can now repeat the simple reasoning involving the harmonic continuation of the Poisson transform to directly infer the ensuing support theorem for $\mathcal{E}^{' [\mathfrak{M}]}(\R^d)$.

\begin{theorem}
\label {support-theorem} 
Let $\mathfrak{M}$ be a weight matrix satisfying $[\mathfrak{M}.1]^{*}_{\mathrm{w}}$, $[\mathfrak{M}.2]'$, and $(NA)$. For each $f \in \mathcal{E}^{' [\mathfrak{M}]}(\R^d)$  there exists a smallest compact set $K \subset \R^d$ such that $f \in  \mathcal{E}^{' [\mathfrak{M}]}(K)$; in fact, $K = \operatorname{supp}_{\mathcal{A}^{'}}f$.
\end{theorem}

It should be noted that Theorem \ref{support-theorem} contains the corresponding support theorem for  $\mathcal{E}^{'[\omega]}(\R^d)$, where  $\omega$ is a  weight function satisfying $(\alpha_0)$ and $\omega(t) = o(t)$, which was earlier obtained by Heinrich and Meise in \cite{Heinrich} via the method from \cite{Hormander85} (without the assumption $(\alpha_0)$). We end this subsection with two remarks.

\begin{remark}
H\"ormander \cite{Hormander85} showed the support theorem for $\mathcal{E}^{'\{M\}}(\R^d)$, where $M$ is a weight sequence satisfying $(M.2)'$ and $(NA)$. His technique can be adapted to show that Theorem \ref{support-theorem} is still valid if one removes the hypothesis $[\mathfrak{M}.1]^{*}_{\mathrm{w}}$ from its statement. We omit details since it is out of the scope of this article.
\end{remark}
\begin{remark}
Suppose that $\mathfrak{M}$ is a non-quasianalytic weight matrix. The assignment $\Omega \mapsto \mathcal{D}^{' [\mathfrak{M}]}(\Omega)$ is a soft sheaf on $\R^d$. For $K \subsetc \R^d$ the space $\{ f \in \mathcal{D}^{' [\mathfrak{M}]}(\R^d) \, | \, \operatorname{supp} f \subseteq K \}$  is canonically isomorphic to $\mathcal{E}^{' [\mathfrak{M}]}(K)$. Hence, there exists a unique minimal  $[\mathfrak{M}]$-carrier  for each $f \in \mathcal{E}^{' [\mathfrak{M}]}(\R^d)$, which is well-known to coincide with $\operatorname{supp}_{\mathcal{A}^{'}} f$ (cf.\ \cite[Lemma 7.4]{Komatsu}), a fact that also follows from Theorem \ref{main-theorem-nqa} below.
\end{remark}

\subsection{Boundary values of harmonic functions in $\mathcal{D}^{'[\mathfrak{M}]}(\Omega)$}  
Let $\mathfrak{M}$ be a non-quasianalytic weight matrix. Let $\Omega \subseteq \R^d$ be open and let $V \subseteq \R^{d+1}$ be open and symmetric with respect to $y$ such that $V \cap \R^d = \Omega$.  Choose  a sequence $(V_n)_{n \in \N}$ of relatively compact open sets in $\R^{d+1}$ that are symmetric with respect to $y$ such that  $V_n \Subset V_{n+1}$ and $V = \bigcup_{n \in \N} V_n$. Set $\Omega_n = V_n \cap \R^d$. We define
\begin{gather*}
\mathcal{H}^{(\mathfrak{M})}_-(V\backslash \Omega) := \varprojlim_{n \in \N} \varinjlim_{M \in \mathfrak{M}}\mathcal{H}^{(M)}_{\infty,-}(V_n\backslash \Omega_n) , \qquad
\mathcal{H}^{\{\mathfrak{M}\}}_-(V\backslash \Omega) := \varprojlim_{n \in \N} \varprojlim_{M \in \mathfrak{M}} \mathcal{H}^{\{M\}}_{\infty,-}(V_n\backslash \Omega_n) .
\end{gather*}
This definition is  independent of the chosen sequence $(V_n)_{n \in \N}$. For two weight matrices $\mathfrak{M}$  and $\mathfrak{N}$ with $\mathfrak{M} [\approx] \mathfrak{N}$ we have that $\mathcal{H}^{[\mathfrak{M}]}_-(V\backslash \Omega) = \mathcal{H}^{[\mathfrak{N}]}_-(V\backslash \Omega)$ as locally convex spaces.

We now show that the elements of $\mathcal{H}^{[\mathfrak{M}]}_-(V\backslash \Omega)$ have boundary values in $\mathcal{D}^{'[\mathfrak{M}]}(\Omega)$.
\begin{proposition}\label{bv-prop-nqa}
Let $\mathfrak{M}$ be a non-quasianalytic weight matrix satisfying $[\mathfrak{M}.1]^{*}_{\mathrm{w}}$ and $[\mathfrak{M}.2]'$. Let $\Omega \subseteq \R^d$ be open and let $V \subseteq \R^{d+1}$ be open and symmetric with respect to $y$ such that $V \cap \R^d = \Omega$. For $F \in \mathcal{H}^{[\mathfrak{M}]}_{-}(V \backslash \Omega)$ we set
$$
\langle \operatorname{bv}(F), \varphi \rangle := \lim_{y \to 0^+} \int_{\R^d} (F(x,y) - F(x,-y)) \varphi(x) \dx, \qquad  \varphi \in \mathcal{D}^{[\mathfrak{M}]}(\Omega).
$$
Then, $\operatorname{bv}(F)$ belongs to  $\mathcal{D}^{'[\mathfrak{M}]}(\Omega)$. 
Moreover, the boundary value mapping
$$
\operatorname{bv}: \mathcal{H}^{[\mathfrak{M}]}_{-}(V \backslash \Omega) \rightarrow  \mathcal{D}^{'[\mathfrak{M}]}(\Omega)
$$
is continuous.
\end{proposition}
\begin{proof}
This can be shown in a similar way to Proposition \ref{bv-theorem} but by using Proposition \ref{almost-harmonic-compact-1} instead of Proposition \ref{almost-harmonic}. 
\end{proof}

Next, we show an ultradistributional version of the Schwarz reflection principle.

\begin{proposition}\label{reflection}
Let $\mathfrak{M}$ be a non-quasianalytic weight matrix satisfying $[\mathfrak{M}.2]'$. Let $\Omega \subseteq \R^d$ be open and let $V \subseteq \R^{d+1}$ be open and symmetric with respect to $y$ such that $V \cap \R^d = \Omega$. Let $F \in \mathcal{H}_{-}(V \backslash \Omega)$ be such that
$$
\lim_{y \to 0^+} \int_{\R^d} (F(x,y) - F(x,-y)) \varphi(x) \dx = 0, \qquad   \varphi \in \mathcal{D}^{[\mathfrak{M}]}(\Omega).
$$
Then, $F$ extends to a harmonic function on $V$.
\end{proposition}
\begin{proof}
Let $\Theta \Subset \Omega$ be arbitrary and choose $r > 0$ such that $\Theta \times (-r,r) \Subset V$. It suffices to show that $F$ extends to a harmonic function on $\Theta \times (-r,r)$. Since $\Delta$ is elliptic, it is enough to show that there is $\widetilde{F} \in  \mathcal{D}^{'[\mathfrak{M}]}(\Theta \times (-r,r))$ such that $\widetilde{F}_{|\Theta \times (-r,r) \backslash \Theta} = F$ and $\Delta \widetilde{F} = 0$ in $ \mathcal{D}^{'[\mathfrak{M}]}(\Theta \times (-r,r))$. To this end, we use the same technique as in \cite[Satz 1.2]{Langenburch}. Let $\varepsilon > 0$ be such that $\Theta \times (-r-\varepsilon,r+\varepsilon) \Subset V$. For $0 < \eta < \varepsilon$ we define $\widetilde{F}^+_\eta, \widetilde{F}^-_\eta \in \mathcal{D}^{'[\mathfrak{M}]}(\Theta \times (-r,r))$ via 
$$
\langle \widetilde{F}^+_\eta, \varphi \rangle := \int_0^r \int_{\Theta} F(x,y+\eta) \varphi(x,y) \dx\dy, \qquad   \varphi \in \mathcal{D}^{[\mathfrak{M}]}(\Theta \times (-r,r)),
$$
and
$$
\langle \widetilde{F}^-_\eta, \varphi \rangle := \int_{-r}^0 \int_{\Theta} F(x,y-\eta) \varphi(x,y) \dx\dy, \qquad   \varphi \in \mathcal{D}^{[\mathfrak{M}]}(\Theta \times (-r,r)).
$$
We claim that  $(\widetilde{F}^\pm_\eta)_{0 < \eta < \varepsilon}$ is a Cauchy net in $\mathcal{D}^{'[\mathfrak{M}]}(\Theta \times (-r,r))$. Before we prove the claim, let us show how it entails the result. Since $\mathcal{D}^{'[\mathfrak{M}]}(\Theta \times (-r,r))$ is complete, there exist $\widetilde{F}^\pm \in \mathcal{D}^{'[\mathfrak{M}]}(\Theta \times (-r,r))$ such that $\lim_{\eta\to0^+} \widetilde{F}^\pm_\eta = \widetilde{F}^\pm$. Set $\widetilde{F} = \widetilde{F}^+ + \widetilde{F}^- \in \mathcal{D}^{'[\mathfrak{M}]}(\Theta \times (-r,r))$. It is clear that $\widetilde{F}_{|\Theta \times (-r,r) \backslash \Theta} = F$. We now show that  $\Delta \widetilde{F} = 0$ in $ \mathcal{D}^{'[\mathfrak{M}]}(\Theta \times (-r,r))$.  Let $\varphi \in  \mathcal{D}^{[\mathfrak{M}]}(\Theta \times (-r,r))$ be arbitrary. Green's theorem (cf.\ the proof of Proposition \ref{bv-theorem}) and the fact that $F$ is odd imply that for all $\varphi \in \mathcal{D}^{[\mathfrak{M}]}(\Theta \times (-r,r))$
\begin{align*}
\langle \Delta \widetilde{F}, \varphi \rangle &= \lim_{\eta \to 0^+}  \int_0^r \int_{\Theta} F(x,y+\eta) \Delta\varphi(x,y) \dx\dy +  \int_{-r}^0 \int_{\Theta} F(x,y-\eta) \Delta\varphi(x,y) \dx\dy \\
&= - \lim_{\eta \to 0^+}  \int_{\Theta} (F(x,\eta) - F(x,-\eta)) \partial_y \varphi(x,0) \dx =0.
\end{align*}
We now show the claim. We only consider $(\widetilde{F}^+_\eta)_{0 < \eta < \varepsilon}$ as  $(\widetilde{F}^-_\eta)_{0 < \eta < \varepsilon}$ can be treated similarly. We have 
\begin{align*}
\lim_{y \to 0^+} \int_{\R^d} \partial^2_y F(x,y) \varphi(x) \dx &= - \lim_{y \to 0^+} \int_{\R^d} \Delta_x F(x,y) \varphi(x) \dx \\
& = -  \lim_{y \to 0^+} \int_{\R^d}  F(x,y) \Delta_x\varphi(x) \dx = 0
\end{align*}
for all $\varphi \in \mathcal{D}^{[\mathfrak{M}]}(\Omega)$. Using the mean-value theorem, we obtain that
$$
\lim_{y \to 0^+} \int_{\R^d} \partial_y F(x,y) \varphi(x) \dx 
$$
exists and is finite for all $\varphi \in \mathcal{D}^{[\mathfrak{M}]}(\Omega)$. Hence, the set $\{ \partial_yF(\, \cdot \,, y) \, | \, 0 < y < r+ \varepsilon\}$ is bounded in $\mathcal{D}'^{[\mathfrak{M}]}(\Theta)$. Let $B$ be an arbitrary bounded subset of $\mathcal{D}^{[\mathfrak{M}]}(\Theta \times (-r,r))$. Then,
$B' = \{ \varphi( \, \cdot \,, y) \, | \,  0 < y <r \}$ is bounded in $\mathcal{D}^{[\mathfrak{M}]}(\Theta)$. For all $0 < \eta, \eta' < \varepsilon$ it holds that
\begin{align*}
\sup_{\varphi \in B}| \langle \widetilde{F}^+_\eta - \widetilde{F}^+_{\eta'} , \varphi \rangle | &= \sup_{\varphi \in B} \left |\int_0^r \int_{\Theta} F(x,y+\eta) - F(x,y+\eta') \varphi(x,y) \dx\dy \right | \\
& = \sup_{\varphi \in B} \left |\int_{\eta'}^\eta \int_0^r \int_{\Theta} \partial_yF(x,y+\lambda) \varphi(x,y) \dx\dy \dlambda \right | \\
&\leq r |\eta - \eta'| \sup_{\psi \in B'} \sup_{0 < y < r+\varepsilon} \left|\int_{\Theta} \partial_yF(x,y) \psi(x) \dx \right|.
\end{align*}
This proves the claim.
\end{proof}

We can now give the representation of  $\mathcal{D}^{'[\mathfrak{M}]}(\Omega)$ by boundary values of harmonic functions.
\begin{theorem}\label{main-theorem-nqa} Let $\mathfrak{M}$ be a non-quasianalytic weight matrix satisfying $[\mathfrak{M}.1]^{*}_{\mathrm{w}}$ and $[\mathfrak{M}.2]'$. Let $\Omega \subseteq \R^d$ be open and let $V \subseteq \R^{d+1}$ be open and symmetric with respect to $y$ such that $V \cap \R^d = \Omega$. Then, the sequence
\begin{equation}
0 \longrightarrow \mathcal{H}_-(V) \longrightarrow  \mathcal{H}^{[\mathfrak{M}]}_-(V \backslash \Omega) \xrightarrow{\phantom,\operatorname{bv}\phantom,}  \mathcal{D}^{'[\mathfrak{M}]}(\Omega) \longrightarrow 0
\label{exact}
\end{equation}
is exact and the boundary value mapping is a topological homomorphism. 
\end{theorem}
\begin{proof}
The boundary value mapping is well-defined and continuous by Proposition \ref{bv-prop-nqa}, while Proposition \ref{reflection} yields that  $\ker \operatorname{bv} = \mathcal{H}_-(V)$. Next, we show that the boundary value mapping is surjective. To this end, we shall use some basic facts about the derived projective limit functor (see the book \cite{Wengenroth} for more information). Choose a sequence $(V_n)_{n \in \N}$ of relatively compact open sets in $\R^{d+1}$ such that $V_n \Subset V_{n+1}$, $V_{n+2} \backslash \overline{V}_n$ has no connected component that is relatively compact in $V_{n+2}$ and $V = \bigcup_{n \in \N} V_n$. Set $\Omega_n = V_n \cap \R^d$. We need to show that the mapping
$$
\operatorname{bv}: \mathcal{H}^{[\mathfrak{M}]}_-(V \backslash \Omega) = \operatorname{Proj} \, (\mathcal{H}^{[\mathfrak{M}]}_-(V_n \backslash \Omega_n))_{n \in \N} \rightarrow  \mathcal{D}^{'[\mathfrak{M}]}(\Omega) = \operatorname{Proj} \, (\mathcal{D}^{'[\mathfrak{M}]}(\Omega_n))_{n \in \N} 
$$
is surjective.
 Consider the following spectrum of short exact sequences
\begin{center}
\begin{tikzpicture}
  \matrix (m) [matrix of math nodes, row sep=2em, column sep=2em]
    {
    0 & \mathcal{H}_-(V_1) & \mathcal{H}^{[\mathfrak{M}]}_-(V_1 \backslash \Omega_1)  & \mathcal{D}^{'[\mathfrak{M}]}(\Omega_1)   \\
   0 & \mathcal{H}_-(V_2) & \mathcal{H}^{[\mathfrak{M}]}_-(V_2\backslash \Omega_2) & \mathcal{D}^{'[\mathfrak{M}]}(\Omega_2)   \\ 
        & \vdots & \vdots & \vdots   \\ };
  
  { [start chain] \chainin (m-1-1);
\chainin (m-1-2);
\chainin (m-1-3)[join={node[above,labeled] {}}];
\chainin (m-1-4)[join={node[above,labeled] {\operatorname{bv}}}];

}
  
  { [start chain] \chainin (m-2-1);
\chainin (m-2-2);
\chainin (m-2-3)[join={node[above,labeled] {}}];
\chainin (m-2-4)[join={node[above,labeled] {\operatorname{bv}}}];
}

{ [start chain] \chainin (m-3-2);
\chainin (m-2-2);
\chainin (m-1-2);
 }

{ [start chain] \chainin (m-3-3);
\chainin (m-2-3);
\chainin (m-1-3);
 }

{ [start chain] \chainin (m-3-4);
\chainin (m-2-4);
\chainin (m-1-4);
 }
\end{tikzpicture}
\end{center}
The Runge approximation theorem for harmonic functions \cite[Theorem 2.3]{Komatsu1991} (see also \cite[Theorem 4.4.5]{Hormander}) and \cite[Theorem 3.2.1]{Wengenroth} yield that $\operatorname{Proj}^1 \, (\mathcal{H}_-(V_n))_{n \in \N} = 0$. By \cite[Proposition 3.1.8]{Wengenroth}, it  therefore suffices to show that for each $f \in \mathcal{D}^{'[\mathfrak{M}]}(\Omega_{n+1})$ there is $F \in \mathcal{H}^{[\mathfrak{M}]}_-(V_n \backslash \Omega_n)$ such that $\operatorname{bv}(F) = f_{|\Omega_n}$. Choose $\chi \in \mathcal{D}^{[\mathfrak{M}]}(\Omega_{n+1})$ such that $\chi \equiv 1$ on $\Omega_n$. Then, $\chi f \in \mathcal{E}^{'[\mathfrak{M}]}(\overline{\Omega}_{n+1})$. Set $F = P[\chi f]$. Theorem \ref{main-theorem} implies that  $F \in \mathcal{H}^{[\mathfrak{M}]}_-(\R^{d+1} \backslash \overline{\Omega}_{n+1}) \subseteq \mathcal{H}^{[\mathfrak{M}]}_-(V_n \backslash \Omega_n)$ and $\operatorname{bv}(F) = f_{|\Omega_n}$. Finally, the boundary value mapping is a topological homomorphism by De Wilde's open mapping theorem.
\end{proof}
\begin{remark}\label{comparison}
Theorem \ref{main-theorem-nqa} particularly applies to $\mathcal{D}^{'[M]}(\Omega)$, where $M$ is a non-quasianalytic weight sequence satisfying $(M.2)'$ and $(M.1)^{*}_{\mathrm{w}}$. For weight sequences satisfying the more restrictive assumptions $(M.2)$ and $(M.3)$ \cite{Komatsu} this result also follows from the work of Komatsu \cite{Komatsu1991}. He had to assume these stronger conditions because he employed the parametrix method.
\end{remark}

 Finally, we can also represent $\mathcal{D}^{'[\omega]}(\Omega)$ via boundary values of harmonic functions. Let $\omega$ be a non-quasianalytic weight function. Let $\Omega \subseteq \R^d$ be open and let $V \subseteq \R^{d+1}$ be open and symmetric with respect to $y$ such that $V \cap \R^d = \Omega$. The space $\mathcal{H}^{[\mathfrak{\omega}]}_-(V \backslash \Omega)$ is defined in the natural way.
\begin{corollary}
Let $\omega$ be a non-quasianalytic weight function satisfying $(\alpha_0)$. Let $\Omega \subseteq \R^d$ be open and let $V \subseteq \R^{d+1}$ be open and symmetric with respect to $y$ such that $V \cap \R^d = \Omega$. Then, the sequence
$$
0 \longrightarrow \mathcal{H}_-(V) \longrightarrow  \mathcal{H}^{[\omega]}_-(V \backslash \Omega) \xrightarrow{\phantom,\operatorname{bv}\phantom,}  \mathcal{D}^{'[\omega]}(\Omega) \longrightarrow 0
$$
is exact and the boundary value mapping is a topological homomorphism.
\end{corollary}
\begin{proof}
This follows from Lemma \ref{reduction}, Lemma \ref{reduction-1},  and Theorem \ref{main-theorem-nqa}.
\end{proof}







\end{document}